\documentclass[11pt,letterpaper]{article}
\usepackage[margin=0.9in]{geometry}
\usepackage[T1]{fontenc}
\usepackage{lmodern}
\usepackage{amsmath,amssymb,amsfonts,amsthm,bm}
\usepackage{graphicx,booktabs,xcolor,subcaption}
\usepackage{algorithm,algpseudocode}
\usepackage[numbers,sort&compress]{natbib}
\usepackage[colorlinks=true,linkcolor=blue,citecolor=blue,urlcolor=blue]{hyperref}
\usepackage{float}
\graphicspath{{./}{./figures/}{./template_images/}}
\newcommand{\R}{\mathbb{R}}
\newcommand{\dd}{\,\mathrm{d}}
\newcommand{\argmin}{\operatorname*{arg\,min}}
\newcommand{\argmax}{\operatorname*{arg\,max}}

\newtheorem{lemma}{Lemma}
\newtheorem{proposition}{Proposition}
\newtheorem{corollary}{Corollary}

\let\originalincludegraphics\includegraphics
\renewcommand{\includegraphics}[2][]{%
  \IfFileExists{#2}{\originalincludegraphics[#1]{#2}}{%
    \IfFileExists{figures/#2}{\originalincludegraphics[#1]{figures/#2}}{%
      \fbox{\parbox{0.75\linewidth}{\centering Missing figure: \texttt{\detokenize{#2}}}}%
    }%
  }%
}
\title{Hamiltonian Primal--Dual Policy Reinforcement Learning\\for Derivative-Free Model-Based Control}
\author{K. Rabiei$^1$ \and Z. Rui$^2$ \and X. Lu$^2$ \and W. Hao$^1$\\[0.5em]
\small $^1$Department of Mathematics, The Pennsylvania State University, USA\\
\small $^2$Department of Computer Science and Engineering, The Pennsylvania State University, USA\\
\small Correspondence: W. Hao (wxh64@psu.edu)}
\date{Preprint manuscript -- October 2026}
\begin{document}
\maketitle
\begin{center}\textbf{Manuscript under review}\end{center}
\begin{abstract}
Finite-horizon optimal control becomes difficult when a dynamical model can be simulated reliably but its derivatives are unavailable, discontinuous, or unreliable. Hamiltonian Primal--Dual Policy Optimization (HPDPO) is introduced as a derivative-free, model-based reinforcement-learning framework for this setting. The control is represented as a time-indexed finite-horizon policy. Candidate actions are evaluated by deterministic suffix rollouts of the supplied transition model, while a regularized Hamiltonian dual formulation supplies structured future-cost information through block-coordinate multiplier updates. No differentiation of the state dynamics or learned critic is required. The method is validated on nonsmooth problems with exact solutions, an analytical and learned ReLU biomarker cascade, an amyloid Neural ODE trained from longitudinal ADNI data, and three multidimensional-control problems with up to five simultaneous controls. Across these examples, HPDPO closely reproduces exact or reference objective values, trajectories, and control structures. The results show that forward simulation and Hamiltonian dual structure can be combined to solve deterministic finite-horizon control problems involving nonsmooth, learned, or black-box dynamics.
\end{abstract}

\section{Introduction}
\label{sec:introduction}

Reinforcement learning (RL) and optimal control both study sequential
decision-making, but they typically differ in the information available and
the objects used to improve a decision rule. Optimal control usually begins
with a dynamical model and derives conditions involving the state, costate, and
control \cite{Pontryagin1962}. RL instead formulates the problem in terms of
states, actions, transitions, rewards, and policies \cite{SuttonBarto2018}. We
consider the finite-horizon problem
\begin{equation}
\label{eq}
\begin{aligned}
\min_{u(\cdot)}\quad
J(u)
&=
\Phi\bigl(x(T)\bigr)
+
\int_0^T \ell\bigl(x(t),u(t),t\bigr)\,\dd t,
\\
\text{subject to}\quad
\dot{x}(t)
&=
f\bigl(x(t),u(t),t\bigr),
\qquad
x(0)=x_0,
\qquad
u(t)\in\mathcal U.
\end{aligned}
\end{equation}
Here $x(t)\in\R^n$ is the state vector and
$u(t)\in\mathcal U\subseteq\R^m$ is the control vector. The scalar-control
case is recovered by taking $m=1$.

Our interest is the setting in which the dynamics can be evaluated and
integrated forward, but their derivatives are unavailable, discontinuous, or
numerically unreliable. This situation arises in nonsmooth analytical models,
black-box simulators, and learned or data-driven dynamical systems, including
models learned from longitudinal observations in biomedical applications
\cite{ZhangLin2018,ZhuWu2021}. For example, ReLU networks are continuous but
not classically differentiable at activation boundaries \cite{Glorot2011}.
A ReLU Neural ODE may therefore provide an accurate transition model without
supplying the smooth Jacobians required by a classical adjoint calculation
\cite{Chen2018}. Similar difficulties arise when a simulator returns
trajectories but does not expose derivatives \cite{Conn2009,Larson2019}.

Classical Hamiltonian optimal control introduces
\begin{equation}
\label{eq:intro-hamiltonian}
H(x,u,\lambda,t)
=
\ell(x,u,t)+\lambda^T f(x,u,t),
\end{equation}
where $\lambda(t)$ is the costate. Forward--backward methods propagate the
state from the initial condition, propagate the costate backward from a
terminal condition, and update the control through a Hamiltonian optimality
condition \cite{LenhartWorkman2007,McAsey2012}. These methods use the model
efficiently, but the costate equation requires derivatives of the dynamics. Reinforcement learning offers a complementary approach in which actions are
evaluated according to their effect on cumulative future outcomes
\cite{SuttonBarto2018}. The specific RL concepts needed for the proposed method
are introduced in Section~\ref{sec:background}. After time discretization of the state equations, Eq.~\eqref{eq} has the
deterministic transition form
\begin{equation}
\label{eq:discrete-transition}
x_{k+1}=F_k(x_k,u_k),
\end{equation}
where, for the RK4 scheme used in the numerical experiments,
\begin{equation}
\label{eq:Rk4-transition}
F_k(x_k,u_k)
=
x_k+\frac{\Delta t}{6}
\left(K_1+2K_2+2K_3+K_4\right),
\end{equation}
with
\[
\begin{aligned}
K_1 &= f(x_k,u_k,t_k),\\
K_2 &= f\left(x_k+\frac{\Delta t}{2}K_1,u_k,
t_k+\frac{\Delta t}{2}\right),\\
K_3 &= f\left(x_k+\frac{\Delta t}{2}K_2,u_k,
t_k+\frac{\Delta t}{2}\right),\\
K_4 &= f\left(x_k+\Delta t K_3,u_k,t_k+\Delta t\right).
\end{aligned}
\]
More generally, $F_k$ may represent any numerical integrator, simulator,
or learned transition model available for forward evaluation. This
representation defines a deterministic finite-horizon decision process in
which candidate actions can be evaluated by propagating their future
consequences through the available model
\cite{Bellman1957,Puterman1994,Sutton1991}.

The present setting lies between standard Hamiltonian methods and commonly
used RL approaches. Hamiltonian methods recover future sensitivity through a
derivative-based costate equation. Model-free policy-gradient and actor--critic
methods avoid derivatives of the physical dynamics, but generally rely on
sampled returns, value functions, or critics
\cite{Williams1992,KondaTsitsiklis2000,Silver2014,Schulman2017,Haarnoja2018}.
Dynamic programming uses a model but seeks a value function over the relevant
state space. Model-based policy-search and predictive-control methods evaluate
future trajectories, but often optimize an entire action sequence or a
parameterized feedback policy through stochastic search or differentiation
through a smooth model
\cite{KamtheDeisenroth2018,Chua2018,Janner2019}. These considerations lead to the
question considered here:

\begin{center}\itshape
Can a finite-horizon policy be improved directly through deterministic model
rollouts while retaining Hamiltonian primal--dual structure, without learning
a critic or differentiating the dynamics?
\end{center}

We address this question with \emph{Hamiltonian primal--dual policy
optimization} (HPDPO). For a prescribed initial state, the decision rule is
represented by the time-indexed sequence
\begin{equation}
\label{eq:intro-policy-sequence}
U=(u_0,u_1,\ldots,u_{M-1})\in\mathcal U^M,
\qquad u_k\in\R^m.
\end{equation}
At time $t_k$, HPDPO evaluates a candidate action by inserting it into the
current sequence and rerolling the model from $t_k$ to the terminal time, namely, $u_k\approx u(t_k)$. The
suffix rollout provides the downstream cost associated with that candidate,
while a regularized Hamiltonian term incorporates a bounded multiplier
trajectory. Accepted actions and their corresponding trajectories are retained
sequentially.  The main contributions are as follows:
\begin{enumerate}
\item We formulate the time-discretized finite-horizon control problem as a
deterministic model-based policy-optimization problem while separating the
role of discretization from the policy-update mechanism.
\item We derive a regularized Hamiltonian primal--dual objective in which
the control sequence is the primal variable, the multiplier sequence is
the dual variable, and the state trajectory is produced only by the
supplied transition model.
\item We replace derivative-based costate integration with bounded
block-coordinate optimization of the multiplier trajectory. For a fixed
policy and $\beta_{\lambda,2}>0$, the dual objective is strongly concave
with a unique maximizer, and repeated unrelaxed cyclic block-coordinate
sweeps converge to this maximizer.
\item We evaluate HPDPO on nonsmooth problems with exact solutions,
analytical and learned ReLU cascade models, an amyloid Neural ODE
trained on longitudinal ADNI data, and three additional multidimensional-control
benchmarks with simultaneous control components, including a five-control
problem solved with a dimension-adaptive QMC action search.
\end{enumerate}

The remainder of the paper develops the connection between Hamiltonian optimal
control and model-based RL, derives HPDPO, and examines its numerical behavior
on analytical and learned dynamical systems.

\section{From Hamiltonian Optimal Control to Model-Based Reinforcement Learning}
\label{sec:background}

This section summarizes the two forms of future-cost information used in
optimal control and RL and establishes the discrete decision problem from
which HPDPO is derived.

\subsection{Hamiltonian optimality and derivative-based future information}
\label{subsec:hamiltonian-fbs}

Consider the finite-horizon problem in Eq.~\eqref{eq}, with prescribed
initial state $x_0$. When the required derivatives exist, Pontryagin's
principle gives the first-order necessary conditions
\cite{Pontryagin1962,LenhartWorkman2007}. Using the Hamiltonian in
Eq.~\eqref{eq:intro-hamiltonian}, these conditions are
\begin{align}
\dot{x}(t)
&= f\bigl(x(t),u(t),t\bigr),
\label{eq:pontryagin-state}
\\
\dot{\lambda}(t)
&= -\frac{\partial H}{\partial x}
   = -\frac{\partial \ell}{\partial x}
   - \left(\frac{\partial f}{\partial x}\right)^{\mathsf T}
     \lambda(t),
\qquad
\lambda(T)=\nabla_x\Phi\bigl(x(T)\bigr),
\label{eq:pontryagin-costate}
\\
u^*(t)
&\in
\operatorname*{arg\,min}_{v\in\mathcal U}
H\bigl(x^*(t),v,\lambda(t),t\bigr).
\label{eq:pontryagin-control}
\end{align}
The state condition is imposed at the beginning of the horizon, whereas the
costate condition is imposed at the end. A forward--backward sweep therefore
starts from a control estimate, propagates the state forward, propagates the
costate backward, and then updates the control
\cite{LenhartWorkman2007,McAsey2012}:
\begin{equation}
\label{eq:fbs-flow}
U^{(r)}
\longrightarrow
X^{(r)}\ \text{forward}
\longrightarrow
\Lambda^{(r)}\ \text{backward}
\longrightarrow
U^{(r+1)}.
\end{equation}
The costate represents the sensitivity of the remaining objective along the current trajectory, but its backward propagation depends on derivatives of the dynamics  $\partial f/\partial x$ in
Eq.~\eqref{eq:pontryagin-costate}. HPDPO instead obtains future-cost information through forward model evaluations.

\subsection{Control problem as a finite-horizon decision process}
\label{subsec:oc-mdp}

Let
\begin{equation}
\label{eq:time-grid}
0=t_0<t_1<\cdots<t_M=T,
\qquad
\Delta t=t_{k+1}-t_k,
\end{equation}
be a uniform time grid. Applying a numerical integrator to the state equation
produces the transition map in
Eq.~\eqref{eq:discrete-transition}. The corresponding finite-horizon
Markov decision process is obtained by identifying
\begin{equation}
\label{eq:rl-oc-correspondence}
\begin{aligned}
S_k&=x_k,\qquad A_k=u_k,\\
c_k(x_k,u_k)&=\Delta t\,\ell(x_k,u_k,t_k),\qquad
c_M(x_M)=\Phi(x_M).
\end{aligned}
\end{equation}
The discretized objective is
\begin{equation}
\label{eq:discrete-control-cost}
J_h(U)
=
\sum_{k=0}^{M-1}c_k(x_k,u_k)+c_M(x_M),
\end{equation}
where the states are generated recursively by $F_k$. For a general feedback
policy, the action has the form $u_k=\mu_k(x_k)$. This paper considers a
prescribed initial state and the restricted policy class
\begin{equation}
\label{eq:restricted-policy-class}
\mu_k(x)\equiv u_k,
\qquad
k=0,\ldots,M-1.
\end{equation}
Thus, the sequence $U=(u_0,\ldots,u_{M-1})$ in
Eq.~\eqref{eq:intro-policy-sequence} is a deterministic time-indexed policy for
the fixed-initial-state problem. Although the policy is restricted to this
class, the decisions remain dynamically coupled: changing $u_k$ changes
$x_{k+1}$ and therefore all subsequent states and costs. Candidate actions
must consequently be compared through their downstream trajectories rather
than through the immediate stage cost alone.

\subsection{Reinforcement learning, value functions, and policy improvement}
\label{subsec:rl-values}

Reinforcement learning provides a language and a set of algorithms for this
sequential problem. At stage $k$, an agent observes state
$S_k$, selects an action $A_k$, receives a reward, and moves to $S_{k+1}$.
Because the present paper is formulated as minimization, we use costs rather
than rewards; taking
$r_k=-c_k$ makes reward maximization and cost minimization equivalent. A
finite-horizon deterministic policy is
\begin{equation}
\label{eq:policy-definition}
    \mu=(\mu_0,\mu_1,\ldots,\mu_{M-1}),
    \qquad
    A_k=\mu_k(S_k).
\end{equation}

For a fixed policy $\mu$, the value function is the remaining cost from state
$x$ at stage $k$:
\begin{equation}
\label{eq:value-function-definition}
\begin{aligned}
    V_k^{\mu}(x)
    &=
    c_k\bigl(x,\mu_k(x)\bigr)
    +
    V_{k+1}^{\mu}
    \left(F_k\bigl(x,\mu_k(x)\bigr)\right),\\
    V_M^{\mu}(x)&=c_M(x).
\end{aligned}
\end{equation}
The action-value function separates the current action from the policy used
thereafter:
\begin{equation}
\label{eq:action-value-definition}
    Q_k^{\mu}(x,a)
    =
    c_k(x,a)
    +
    V_{k+1}^{\mu}\bigl(F_k(x,a)\bigr).
\end{equation}
It therefore answers the question that motivates the present paper: what is the
remaining cost if action $a$ is applied now and the current policy is followed
for the rest of the horizon? Once this quantity is available, a policy can be
improved by choosing an action with smaller action value,
\begin{equation}
\label{eq:policy-improvement-basic}
    \mu_k^{+}(x)
    \in
    \argmin_{a\in\mathcal A(x)}
    Q_k^{\mu}(x,a).
\end{equation}
The corresponding optimal Bellman recursion is
\begin{equation}
\label{eq:bellman-optimality}
\begin{aligned}
    V_k^{*}(x)
    &=
    \min_{a\in\mathcal A(x)}
    \left[
        c_k(x,a)
        +V_{k+1}^{*}\bigl(F_k(x,a)\bigr)
    \right],\\
    V_M^{*}(x)&=c_M(x).
\end{aligned}
\end{equation}

Equations~\eqref{eq:action-value-definition}--\eqref{eq:bellman-optimality}
make the role of reinforcement learning explicit. A current action is not
ranked only by its immediate cost; it is ranked by immediate cost plus future
cost. Reinforcement-learning methods differ primarily in how this future-cost
information is represented, estimated, or computed. Value-based methods learn
or approximate $V$ or $Q$; policy-based methods optimize a policy directly;
and actor--critic methods combine a policy with a learned value approximation
\cite{WatkinsDayan1992,Williams1992,KondaTsitsiklis2000,Silver2014,
Schulman2017,Haarnoja2018}.

\subsection{Model-based reinforcement learning by deterministic suffix rollout}
\label{subsec:rl-rollout}

The preceding formulation also depends on whether the transition model is used
explicitly. Model-free reinforcement learning estimates decision quality from
sampled interaction without planning through a supplied model. Model-based
reinforcement learning uses an analytical, learned, or simulator transition
model to predict future states and improve the policy
\cite{Sutton1991,DeisenrothRasmussen2011,KamtheDeisenroth2018,Chua2018,
Janner2019}. The present setting is naturally model-based because $F_k$ is
available and can be evaluated forward.

Let $U$ be the current time-indexed policy. At coordinate $k$, replace the
current action by a candidate $v$:
\begin{equation}
\label{eq:rollout-candidate-policy}
    U^{(k)}(v)
    =
    (u_0,\ldots,u_{k-1},v,u_{k+1},\ldots,u_{M-1}).
\end{equation}
The candidate produces the suffix trajectory
\begin{equation}
\label{eq:suffix-rollout}
\begin{aligned}
    x_{k+1}^{(v)}
    &=F_k(x_k,v),\\
    x_{j+1}^{(v)}
    &=F_j(x_j^{(v)},u_j),
    \qquad j=k+1,\ldots,M-1.
\end{aligned}
\end{equation}
Its complete downstream cost is
\begin{equation}
\label{eq:rollout-action-value}
\begin{aligned}
    Q_k^{\mathrm{roll}}(x_k,v;U)
    =
    c_k(x_k,v)
    +
    \sum_{j=k+1}^{M-1}c_j(x_j^{(v)},u_j)
    +
    c_M(x_M^{(v)}).
\end{aligned}
\end{equation}
An ordinary rollout policy-improvement step is then
\begin{equation}
\label{eq:rollout-improvement}
    u_k^{+}
    \in
    \argmin_{v\in\mathcal U}
    Q_k^{\mathrm{roll}}(x_k,v;U).
\end{equation}

This rollout evaluates a candidate action through the supplied transition
model and uses the resulting future cost to improve the policy. The quantity in
Eq.~\eqref{eq:rollout-action-value} supplies, at the currently visited state,
the information represented by an action-value function. It is not a learned
or stored approximation of $Q_k^{\mu}(x,a)$ over the full state--action space.
Instead, the needed action values are recomputed directly by deterministic
rollout whenever candidate actions are compared. This distinction is central
to HPDPO: the method uses reinforcement-learning policy improvement, but it
does not require a global critic.

\subsection{Bellman and Hamiltonian routes to future-cost information}
\label{subsec:bellman-hamiltonian}

The reinforcement-learning value function and the optimal-control costate arise
from the same underlying requirement: current decisions must be evaluated by
their effect on the remaining objective. Their connection is most transparent
in continuous time. Define the optimal future cost
\begin{equation}
\label{eq:continuous-value-function}
    V(t,x)
    =
    \inf_{u(\cdot)}
    \left\{
        \Phi(x(T))
        +
        \int_t^T\ell(x(s),u(s),s)\,\dd s
    \right\}.
\end{equation}
Under standard smoothness assumptions, $V$ satisfies the
Hamilton--Jacobi--Bellman equation
\begin{equation}
\label{eq:hjb-equation}
    -V_t(t,x)
    =
    \min_{u\in\mathcal U}
    \left[
        \ell(x,u,t)
        +
        \nabla_xV(t,x)^T f(x,u,t)
    \right]
\end{equation}
\cite{Bellman1957,Bertsekas2017,LewisVrabie2009,KimYang2020}. Along a
sufficiently smooth optimal trajectory,
\begin{equation}
\label{eq:costate-value-relation}
    \lambda^*(t)
    =
    \nabla_xV(t,x^*(t)).
\end{equation}
Thus, the costate and the value gradient both encode future sensitivity, while
$Q$ evaluates the future consequence of a particular action.

This relationship clarifies the difference between FBS and the proposed
method. Both may use the same time grid and the same control vector $U$; the
use of discretization alone does not determine the update mechanism. FBS first
computes a complete state trajectory and a complete derivative-based costate
trajectory, then updates the control vector. HPDPO instead considers the policy
coordinates sequentially. Symbolically,
\begin{equation}
\label{eq:sequential-state-action-flow}
    x_k
    \longrightarrow
    u_k
    \longrightarrow
    x_{k+1}
    \longrightarrow
    u_{k+1}.
\end{equation}

Building on this rollout evaluation, HPDPO adds a Hamiltonian dual trajectory
updated through bounded block-coordinate maximization. The method does not solve the HJB equation, does
not learn a global $V$ or $Q$, and does not assume that its regularized discrete
multiplier is identical to $\nabla_xV$. Its most precise classification is
\emph{deterministic model-based direct policy optimization in a restricted
finite-horizon policy class}.  The required method must
therefore combine two sources of information: the model-based RL rollout,
which gives the downstream cost generated by a candidate action, and the
Hamiltonian dual information, which gives the multiplier contribution
along that rollout.
Section~\ref{sec:proposed-method} derives HPDPO by coupling these two
components in a regularized finite-dimensional primal--dual objective.


\section{Hamiltonian Primal--Dual Policy Optimization}
\label{sec:proposed-method}

Building on the rollout formulation of Section~\ref{sec:background}, this section derives the regularized Hamiltonian primal--dual objective and the corresponding policy and multiplier updates.

\subsection{Integrated-by-parts Hamiltonian representation}
\label{subsec:continuous-ibp}

For the dynamics in Eq.~\eqref{eq}, introduce a multiplier trajectory
$\lambda(t)\in\mathbb R^n$. The dynamics-augmented objective is
\begin{equation}
\label{eq:continuous-lagrangian}
\begin{aligned}
    \mathcal L(x,u,\lambda)
    =
    \Phi(x(T))
    +
    \int_0^T
    \left[
        \ell(x,u,t)
        +
        \lambda(t)^T\bigl(f(x,u,t)-\dot x(t)\bigr)
    \right]
    \,\dd t.
\end{aligned}
\end{equation}
When the state equation is satisfied, the residual term vanishes and
$\mathcal L$ reduces to the original objective. Integrating the term containing
$\dot x$ by parts gives
\begin{equation}
\label{eq:ibp-identity}
\begin{aligned}
    -\int_0^T\lambda(t)^T\dot x(t)\,\dd t
    =
    -\lambda(T)^Tx(T)
    +\lambda(0)^Tx(0)
    +\int_0^T\dot\lambda(t)^Tx(t)\,\dd t.
\end{aligned}
\end{equation}
Substitution into Eq.~\eqref{eq:continuous-lagrangian} yields
\begin{equation}
\label{eq:continuous-ibp}
\begin{aligned}
    \mathcal L_{\mathrm{IBP}}(x,u,\lambda)
    ={}&
    \Phi(x(T))
    +
    \int_0^T
    \left[
        \ell(x,u,t)
        +\lambda(t)^Tf(x,u,t)
        +\dot\lambda(t)^Tx(t)
    \right]
    \,\dd t\\
    &-\lambda(T)^Tx(T)
    +\lambda(0)^Tx(0).
\end{aligned}
\end{equation}
This representation is the continuous primal--dual basis of HPDPO. The
derivative is shifted to the multiplier trajectory, which can be discretized
and updated directly through block-coordinate maximization.

\subsection{Rollout-generated discrete primal--dual objective}
\label{subsec:discrete-objective}

Using the grid in Eq.~\eqref{eq:time-grid}, let
\begin{equation}
    \lambda_k\approx\lambda(t_k),
    \qquad
    \Lambda=(\lambda_0,\lambda_1,\ldots,\lambda_M).
\end{equation}
For a given policy $U$, repeated application of the numerical transition map
generates the state trajectory,
\begin{equation}
\label{eq:rollout-operator}
    X=\mathcal R(U).
\end{equation}
In reinforcement-learning language, $\mathcal R$ is the deterministic environment
model used to evaluate a policy. Using left-endpoint quadrature for the integral terms and a forward difference
for $\dot\lambda$, Eq.~\eqref{eq:continuous-ibp} gives
\begin{equation}
\label{eq:discrete-ibp}
\begin{aligned}
    L_{\mathrm{IBP}}(U,\Lambda)
    ={}&
    \sum_{k=0}^{M-1}
    \Big[
        \Delta t\,\ell(x_k,u_k,t_k)
        +\Delta t\,\lambda_k^Tf(x_k,u_k,t_k)     
    \\
    &+(\lambda_{k+1}-\lambda_k)^Tx_k\Big] -\lambda_M^Tx_M
    +\lambda_0^Tx_0
    +\Phi(x_M),\\
    \qquad X=\mathcal R(U).
\end{aligned}
\end{equation}
To expose the coordinate structure, define
\begin{equation}
\label{eq:discrete-dual-coefficients-early}
    b_k(U)
    =
    \begin{cases}
        \Delta t\,f(x_0,u_0,t_0), & k=0,\\[1mm]
        \Delta t\,f(x_k,u_k,t_k)+x_{k-1}-x_k,
        &1\leq k\leq M-1,\\[1mm]
        x_{M-1}-x_M, &k=M.
    \end{cases}
\end{equation}
Then,
\begin{equation}
\label{eq:discrete-ibp-compact}
    L_{\mathrm{IBP}}(U,\Lambda)
    =
    J_h(U)+\sum_{k=0}^{M}\lambda_k^Tb_k(U).
\end{equation}
The vectors $b_k(U)$ couple local vector-field evaluations to the trajectory
generated by the numerical rollout. Since the two discretizations are different,
$b_k(U)$ need not vanish even when $X=\mathcal R(U)$.

\subsection{Regularized bounded saddle formulation}
\label{subsec:regularized-saddle}

Because $L_{\mathrm{IBP}}$ is linear in $\Lambda$, unconstrained maximization
may be unbounded or poorly conditioned. We therefore introduce
\begin{equation}
\label{eq:regularized-saddle}
\begin{aligned}
    L_{\mathrm{reg}}(U,\Lambda)
    ={}&
    L_{\mathrm{IBP}}(U,\Lambda)
    +\beta_u\sum_{k=0}^{M-2}\|u_{k+1}-u_k\|^2\\
    &-\beta_{\lambda,s}\sum_{k=0}^{M-1}
        \|\lambda_{k+1}-\lambda_k\|^2
    -\beta_{\lambda,2}\sum_{k=0}^{M}\|\lambda_k\|^2,
\end{aligned}
\end{equation}
where
\begin{equation}
    \beta_u\geq0,
    \qquad
    \beta_{\lambda,s}\geq0,
    \qquad
    \beta_{\lambda,2}\geq0.
\end{equation}
The policy regularization term discourages grid-scale oscillations. The
multiplier smoothness and $L^2$ terms discourage large or rapidly varying
dual sequences. The signs of these terms follow from the saddle structure:
since $U$ is minimized, its regularization term enters with a positive
sign, so oscillatory policies increase the objective. In contrast,
$\Lambda$ is maximized, so its regularization terms enter with negative
signs; large or rapidly varying multiplier sequences then decrease the
objective and are therefore disfavored by the maximization.

The multiplier sequence is restricted to the box
\begin{equation}
\label{eq:lambda-box}
\begin{aligned}
\mathcal B_{\lambda}
=
\bigl\{\Lambda=(\lambda_0,\ldots,\lambda_M):\;&
-\lambda_{\max}\leq\lambda_{k,i}\leq\lambda_{\max},\\
&k=0,\ldots,M,\quad i=1,\ldots,n\bigr\}.
\end{aligned}
\end{equation}
where $\lambda_{\max}>0$. HPDPO is based on the regularized
finite-dimensional saddle problem
\begin{equation}
\label{eq:regularized-saddle-problem}
    \min_{U\in\mathcal U^M}
    \max_{\Lambda\in\mathcal B_{\lambda}}
    L_{\mathrm{reg}}(U,\Lambda).
\end{equation}
The regularization parameters are numerical stabilization parameters and
modify the discrete saddle function.

No convexity or differentiability assumption is imposed on the dependence of
$L_{\mathrm{reg}}$ on the policy $U$, since the trajectory
$X=\mathcal R(U)$ may be generated by nonlinear, nonsmooth, or black-box
dynamics. Thus, the primal dependence may be nonsmooth and nonconvex.
For fixed $U$, however, the dependence on $\Lambda$ is quadratic and
concave, and becomes strongly concave when $\beta_{\lambda,2}>0$, as shown
below. Under this condition, the regularized HPDPO problem therefore has a
possibly nonsmooth and nonconvex primal block together with a strongly
concave dual block.

For the following analysis, the multiplier sequence
$\Lambda
=
\begin{pmatrix}
\lambda_0,
\lambda_1,
\cdots,
\lambda_M
\end{pmatrix}^T
\in\mathbb R^{(M+1)n}$ is the stacked column vector.
To represent the consecutive differences appearing in the multiplier
smoothness penalty, define the first-difference matrix
\begin{equation}
D_M
=
\begin{pmatrix}
-1 & 1 & 0 & \cdots & 0\\
0 & -1 & 1 & \cdots & 0\\
\vdots & & \ddots & \ddots & \vdots\\
0 & \cdots & 0 & -1 & 1
\end{pmatrix}
\in\mathbb R^{M\times(M+1)},
\end{equation}
and its block extension
\begin{equation}
\mathcal D
=
D_M\otimes I_n
\in
\mathbb R^{Mn\times(M+1)n},
\end{equation}
where $I_n$ denotes the $n\times n$ identity matrix. The matrix $D_M$
acts on the time index, while the Kronecker product with $I_n$ applies the
same first-difference operation to the $n$ components of each multiplier
block. Hence,
\begin{equation}
\mathcal D\Lambda
=
\begin{pmatrix}
\lambda_1-\lambda_0\\
\lambda_2-\lambda_1\\
\vdots\\
\lambda_M-\lambda_{M-1}
\end{pmatrix},
\end{equation}
and therefore
\begin{equation}
\|\mathcal D\Lambda\|^2
=
\sum_{k=0}^{M-1}
\|\lambda_{k+1}-\lambda_k\|^2.
\end{equation}
Similarly,
\begin{equation}
\|\Lambda\|^2
=
\sum_{k=0}^{M}\|\lambda_k\|^2.
\end{equation}
Consequently, for fixed $U$, the terms of $L_{\mathrm{reg}}$ that depend on
$\Lambda$ can be written compactly as
\begin{equation}
\label{eq:dual-compact-form}
\sum_{k=0}^{M}\lambda_k^Tb_k(U)
-\beta_{\lambda,s}\|\mathcal D\Lambda\|^2
-\beta_{\lambda,2}\|\Lambda\|^2.
\end{equation}
This representation makes the curvature of the dual subproblem explicit. For a twice-differentiable function, a uniform negative-definite Hessian
bound provides a standard characterization of strong concavity
\cite{BoydVandenberghe2004}.

\begin{lemma}[Strong concavity of the dual block]
\label{lem:dual-strong-concavity}
For every fixed admissible policy $U$, if $\beta_{\lambda,2}>0$, then
$L_{\mathrm{reg}}(U,\cdot)$ is
$2\beta_{\lambda,2}$-strongly concave with respect to the multiplier
sequence $\Lambda$.
\end{lemma}

\noindent\textit{The proof is provided in Supporting Information, Section S1.1.}

\begin{proposition}[Existence and uniqueness of the dual maximizer]
\label{prop:dual-existence-uniqueness}
Let $U$ be any fixed admissible policy and suppose that
$\beta_{\lambda,2}>0$. Then the dual subproblem
\begin{equation}
\max_{\Lambda\in\mathcal B_\lambda}
L_{\mathrm{reg}}(U,\Lambda)
\end{equation}
admits a unique global maximizer, denoted by
\begin{equation}
\Lambda^\star(U)
=
\argmax_{\Lambda\in\mathcal B_\lambda}
L_{\mathrm{reg}}(U,\Lambda).
\end{equation}
\end{proposition}

\noindent\textit{The proof is provided in Supporting Information, Section S1.2.}

\subsection{Hamiltonian rollout action score}
\label{subsec:hamiltonian-rollout-score}

We next connect the saddle objective to reinforcement-learning policy
improvement. At policy coordinate $k$, use the candidate policy
$U^{(k)}(v)$ and suffix rollout defined in
Eqs.~\eqref{eq:rollout-candidate-policy} and \eqref{eq:suffix-rollout}. For a
fixed multiplier sequence, define the Hamiltonian rollout action score
\begin{equation}
\label{eq:discrete-hamiltonian-q}
    \mathcal Q_{k,h}^{\mathrm{HP}}(v\mid U,\Lambda)
    =
    L_{\mathrm{reg}}\bigl(U^{(k)}(v),\Lambda\bigr).
\end{equation}
Because the trajectory prefix before $t_k$ is unchanged, all
candidate-independent terms can be collected into a constant $C_k$. The
candidate-dependent part can be written as
\begin{equation}
\label{eq:discrete-hamiltonian-q-decomposition}
\begin{aligned}
    \mathcal Q_{k,h}^{\mathrm{HP}}(v\mid U,\Lambda)
    =C_k
    &+Q_k^{\mathrm{roll}}(x_k,v;U)\\
    &+\sum_{j=k}^{M}\lambda_j^T
      b_j\bigl(U^{(k)}(v)\bigr)\\
    &+\beta_u\sum_{j=0}^{M-2}
      \left\|u_{j+1}^{(k)}(v)-u_j^{(k)}(v)\right\|^2.
\end{aligned}
\end{equation}
The multiplier regularization is constant during the policy-coordinate update
and is absorbed into $C_k$. The central block-coordinate updates are
\begin{align}
\label{eq:central-policy-update}
u_k^{+}
&\in
\argmin_{v\in\mathcal U}
\mathcal Q_{k,h}^{\mathrm{HP}}(v\mid U,\Lambda),
\\
\label{eq:central-dual-update}
\lambda_k^{+}
&\in
\argmax_{z\in[-\lambda_{\max},\lambda_{\max}]^n}
L_{\mathrm{reg}}
\left(
U^{+},
(\lambda_0,\ldots,\lambda_{k-1},
z,\lambda_{k+1},\ldots,\lambda_M)
\right).
\end{align}
\subsection{Alternating policy and dual coordinate updates}
\label{subsec:coordinate-updates}

Let $(U^{(r)},\Lambda^{(r)})$ denote the current outer iterate.

\paragraph*{Sequential model-based policy sweep.}
During the policy update, $\Lambda^{(r)}$ is fixed and the policy coordinates
are considered in the forward order
\begin{equation}
\label{eq:forward-policy-order}
    u_0\longrightarrow u_1\longrightarrow\cdots\longrightarrow u_{M-1}.
\end{equation}
At coordinate $k$, the actions before $k$ are the values already accepted in
the current sweep, while the later actions retain their current values. For
each candidate $v$, only the suffix $x_{k+1}^{(v)},\ldots,x_M^{(v)}$ must be
recomputed. The candidate is scored by
Eq.~\eqref{eq:discrete-hamiltonian-q}, and the minimizing value is accepted.

For a bounded vector-valued control, the action subproblem at a fixed time
point may be evaluated on a finite candidate set
\begin{equation}
\mathcal V_\delta
=
\{v_1,\ldots,v_Q\}
\subseteq
\mathcal U\subseteq\R^m.
\end{equation}
The candidate vectors are always evaluated jointly in $\mathbb{R}^m$.
The construction of the finite candidate set is chosen according to the
control dimension, as described below. To quantify the resolution of the
candidate set, define its covering radius by
\begin{equation}
\delta
=
\sup_{v\in\mathcal U}
\min_{w\in\mathcal V_\delta}\|v-w\|_2.
\end{equation}
Thus, every admissible action is within the distance $\delta$ of at least one
candidate action.
The following result shows that the finite candidate search consistently
approximates the continuous action minimization as the candidate set is
refined. Such error bounds based on the covering radius of a finite
evaluation set are standard in Lipschitz optimization; see, for example,
\cite{MalherbeVayatis2017}.

\begin{corollary}[Consistency of the finite action search]
\label{cor:finite-action-search}
Fix $U$, $\Lambda$, and a policy coordinate $k$. Suppose that there exists
$L_Q>0$ such that
\begin{equation}
\left|
\mathcal Q_{k,h}^{\mathrm{HP}}(v\mid U,\Lambda)
-
\mathcal Q_{k,h}^{\mathrm{HP}}(w\mid U,\Lambda)
\right|
\leq
L_Q \|v-w\|_2
\end{equation}
for all $v,w\in\mathcal U$. Then
\begin{equation}
0
\leq
\min_{v\in\mathcal V_\delta}
\mathcal Q_{k,h}^{\mathrm{HP}}(v\mid U,\Lambda)
-
\min_{v\in\mathcal U}
\mathcal Q_{k,h}^{\mathrm{HP}}(v\mid U,\Lambda)
\leq
L_Q\delta .
\end{equation}
Consequently, the minimum score over $\mathcal V_\delta$ converges to
the continuous action minimum as $\delta\to0$.
\end{corollary}

\noindent\textit{The proof is provided in Supporting Information, Section S1.3.}

\paragraph*{Dimension-adaptive derivative-free action search.}
The finite-set formulation above is used for both low- and higher-dimensional
controls; only the construction of the candidate set changes with the control
dimension. For small $m$, a Cartesian-product grid gives simple and systematic
coverage of the admissible control set. If $q$ values are used for each control
component, however, the number of joint candidates grows as
\[
q^m,
\]
so a dense tensor grid becomes expensive as $m$ increases. In this case, we
replace the tensor grid by a Sobol quasi-Monte Carlo (QMC) set of joint control
vectors. Independent Monte Carlo sampling could also be used, but finite random
samples may cluster in some regions and leave other regions less well covered,
and they introduce additional run-to-run variation. Sobol points provide a more
regular, reproducible coverage of the bounded control box. QMC sampling can also be used when $m$ is small, but a Cartesian grid is
preferred there when exhaustive structured coverage is still inexpensive.

For box-constrained controls, a minimizer may occur in the interior, on the
boundary, or at a corner of the admissible region. Therefore, in the
higher-dimensional QMC search we augment the Sobol candidates with the $2^m$
corners of the current search box.
Both search strategies use the same coarse-to-fine idea, with a global joint
search followed, when needed, by local joint searches in smaller regions around
the current best action. Thus, Cartesian grids and Sobol/QMC sampling are two
numerical realizations of the same HPDPO action minimization; the rollout,
saddle objective, and primal--dual block updates are unchanged.

The finite candidate evaluation remains an $m$-dimensional action search at
each time point, rather than a grid search over the full policy sequence
$U$ \cite{KoldaLewisTorczon2003,Tseng2001,Wright2015,GrippoSciandrone2000}.
For scalar controls, this reduces to the one-dimensional search used in the
scalar numerical examples below.

After accepting $u_k^{+}$, the corresponding suffix trajectory replaces the
previous one. The next action is therefore evaluated from the updated state,
which gives the sequential model-based policy-improvement structure in
Eq.~\eqref{eq:sequential-state-action-flow}. An optional relaxed update is
\begin{equation}
\label{eq:u-relaxation}
    u_k
    \leftarrow
    (1-\omega_u)u_k+\omega_u u_k^{+},
    \qquad 0<\omega_u\leq1.
\end{equation}
When relaxation is used, the accepted suffix is rerolled using the relaxed
value before the next coordinate is considered.

\paragraph{Backward Hamiltonian dual sweep.}

After the policy sweep, the updated policy $U^{(r+1)}$ and its rollout
$X^{(r+1)}=\mathcal R(U^{(r+1)})$ are fixed. The multiplier blocks are
then maximized in the default backward order
\begin{equation}
\label{eq:backward-dual-order}
\lambda_M
\longrightarrow
\lambda_{M-1}
\longrightarrow
\cdots
\longrightarrow
\lambda_0.
\end{equation}
Assuming $\beta_{\lambda,2}>0$, fix $U$, and hence
$X=\mathcal R(U)$ and $b_k(U)$, and consider an interior multiplier block
$1\leq k\leq M-1$. The terms of
$L_{\mathrm{reg}}$ involving $\lambda_k$ are
\begin{equation}
\begin{aligned}
b_k(U)^T\lambda_k
&-\beta_{\lambda,s}\|\lambda_k-\lambda_{k-1}\|^2\\
&-\beta_{\lambda,s}\|\lambda_{k+1}-\lambda_k\|^2
-\beta_{\lambda,2}\|\lambda_k\|^2 .
\end{aligned}
\end{equation}
Differentiating with respect to $\lambda_k$ gives
\begin{equation}
b_k(U)
+2\beta_{\lambda,s}(\lambda_{k-1}+\lambda_{k+1})
-(4\beta_{\lambda,s}+2\beta_{\lambda,2})\lambda_k.
\end{equation}
Setting this expression equal to zero gives the unconstrained maximizer of
the interior block subproblem. The endpoint expressions follow analogously
by retaining only the single neighboring smoothness term. Consequently,
\begin{equation}
\label{eq:lambda-closed-form}
\widehat{\lambda}_k
=
\begin{cases}
\displaystyle
\frac{
b_0(U)+2\beta_{\lambda,s}\lambda_1
}{
2\beta_{\lambda,s}+2\beta_{\lambda,2}
},
& k=0,
\\[4mm]
\displaystyle
\frac{
b_k(U)+2\beta_{\lambda,s}
(\lambda_{k-1}+\lambda_{k+1})
}{
4\beta_{\lambda,s}+2\beta_{\lambda,2}
},
& 1\leq k\leq M-1,
\\[4mm]
\displaystyle
\frac{
b_M(U)+2\beta_{\lambda,s}\lambda_{M-1}
}{
2\beta_{\lambda,s}+2\beta_{\lambda,2}
},
& k=M.
\end{cases}
\end{equation}

Since the block objective is a concave quadratic centered at
$\widehat{\lambda}_k$, maximizing it over
$[-\lambda_{\max},\lambda_{\max}]^n$ is equivalent to projecting
$\widehat{\lambda}_k$ onto that box. Hence, the exact constrained
block maximizer is
\begin{equation}
\label{eq:lambda-coordinate-update}
\lambda_k^{+}
=
\Pi_{[-\lambda_{\max},\lambda_{\max}]^n}
\left(\widehat{\lambda}_k\right).
\end{equation}
Thus, each multiplier block is updated by its exact constrained maximizer.

\begin{proposition}[Convergence of the dual block-coordinate sweep]
\label{prop:dual-coordinate-convergence}
Fix an admissible policy $U$ and suppose that
$\beta_{\lambda,2}>0$. If the multiplier blocks are repeatedly updated
in a fixed cyclic order using the exact constrained block maximizer in
Eq.~\eqref{eq:lambda-coordinate-update}, then the resulting multiplier
sequence converges to the unique solution $\Lambda^\star(U)$ of the
fixed-$U$ dual maximization problem.
\end{proposition}

\noindent\textit{The proof is provided in Supporting Information, Section S1.4.}

A prescribed terminal condition
$\lambda_M=\nabla_x\Phi(x_M)$ may instead be imposed when appropriate.
In the bounded dual formulation used here, $\lambda_M$ is optimized as part
of the multiplier sequence. Algorithm~\ref{alg:hpdpo} summarizes the
default iteration. We use $\mathcal R_k(x_k,U)$ to denote the suffix rollout from stage $k$
to the terminal time. The forward-policy/backward-dual order reflects the causal propagation of
the state and the backward orientation of dual information in classical
optimal control. The finite-dimensional block-coordinate formulation does
not require this specific order, however. Alternating or reversed
Gauss--Seidel sweeps may also be used
\cite{Wright2015,GrippoSciandrone2000}. In the numerical experiments, the
alternative ordering produced essentially the same policies and state
trajectories, indicating that the reported solutions were not artifacts of
one particular sweep direction.
\begin{algorithm}[t]
\footnotesize
\caption{Hamiltonian primal--dual policy optimization (HPDPO)}
\label{alg:hpdpo}
\begin{algorithmic}[1]
\Require $U^{(0)}$, $\Lambda^{(0)}$, candidate-set rule $\mathcal G$, maximum iterations $R_{\max}$, relaxation factor $\omega_u$, tolerance $\varepsilon$
\Ensure $U^*$, $\Lambda^*$, and $X^*$
\State $X^{(0)}\leftarrow\mathcal R(U^{(0)})$
\For{$r=0,\ldots,R_{\max}-1$}
  \State $\widetilde U\leftarrow U^{(r)}$, $\widetilde\Lambda\leftarrow\Lambda^{(r)}$, $\widetilde X\leftarrow X^{(r)}$
  \Statex \textit{Model-based RL policy-improvement sweep}
  \For{$k=0,\ldots,M-1$}
    \State $\mathcal V\leftarrow\mathcal G(k,r,\widetilde U)$
    \ForAll{$v\in\mathcal V$}
      \State Form $U^{(k)}(v)$ using the accepted prefix and current suffix
      \State Reroll $X^{(k)}(v)=\mathcal R_k(\widetilde x_k,U^{(k)}(v))$
      \State Evaluate $\mathcal Q_{k,h}^{\mathrm{HP}}(v\mid\widetilde U,\widetilde\Lambda)$
    \EndFor
    \State $v_k^*\leftarrow\argmin_{v\in\mathcal V}\mathcal Q_{k,h}^{\mathrm{HP}}(v\mid\widetilde U,\widetilde\Lambda)$
    \Statex \textit{If local refinement is used, repeat the candidate evaluation in a smaller region around $v_k^*$.}
    \State $\widetilde u_k\leftarrow(1-\omega_u)\widetilde u_k+\omega_u v_k^*$
    \State $\widetilde X\leftarrow\mathcal R_k(\widetilde x_k,\widetilde U)$
  \EndFor
  \State $U^{(r+1)}\leftarrow\widetilde U$, $X^{(r+1)}\leftarrow\widetilde X$
  \Statex \textit{Hamiltonian dual sweep}
  \For{$k=M,M-1,\ldots,0$}
    \State Compute $\widehat\lambda_k$ from Eq.~\eqref{eq:lambda-closed-form}
    \State $\lambda_k^+\leftarrow\Pi_{[-\lambda_{\max},\lambda_{\max}]^n}(\widehat\lambda_k)$
    \State $\widetilde\lambda_k\leftarrow\lambda_k^+$
  \EndFor
  \State $\Lambda^{(r+1)}\leftarrow\widetilde\Lambda$
  \If{$\|U^{(r+1)}-U^{(r)}\|_\infty+\|\Lambda^{(r+1)}-\Lambda^{(r)}\|_\infty<\varepsilon$}
    \State \textbf{break}
  \EndIf
\EndFor
\State $U^*\leftarrow U^{(r+1)}$, $\Lambda^*\leftarrow\Lambda^{(r+1)}$, $X^*\leftarrow X^{(r+1)}$
\end{algorithmic}
\end{algorithm}


\section{Numerical examples and Validation Studies}
\label{sec:numerical-results}

We first evaluate HPDPO in four numerical examples. Examples~1
and~2 compare the numerical solution with exact solutions of nonsmooth scalar problems. Example~3 considers a four-state nonlinear ODE cascade model of Alzheimer's disease. The corresponding optimal-control problem is first solved using both the forward--backward sweep (FBS) method and HPDPO. A neural network with ReLU activation functions is then trained to approximate the analytical right-hand side of the ODE system. Because the resulting learned dynamics are generally nondifferentiable, the neural-network-based optimal-control problem is solved using HPDPO. The control and state trajectories obtained from the analytical FBS model, the analytical HPDPO model, and the learned ReLU HPDPO model are then compared. Example~4 applies the proposed method to a fully data-driven setting. After preprocessing the ADNI data and normalizing the amyloid-$\beta$ measurements to the interval $[0,1]$, we train a ReLU neural ODE on the longitudinal observations using the disease progression score as the independent variable. The resulting learned system is then used as the state equation in an optimal-control problem, which is solved with HPDPO. To assess consistency with an adjoint-based approach, the learned neural right-hand side is approximated by a smooth Legendre expansion, and the corresponding optimal-control problem is solved using the forward--backward sweep method as a reference.

Two additional examples are then used to evaluate the multidimensional action
search. Multidimensional Example~1 is a two-state compartmental system with two
simultaneous controls and a known exact switching solution. Multidimensional
Example~2 is a five-state nonlinear Alzheimer's disease model with two
simultaneous treatment controls, for which the HPDPO result is compared with
the numerical solution reported in the original model study.

\subsection{Example 1:  Nonsmooth Problem with an Exact Solution}
\label{sec:constructed-nonsmooth-benchmark}

We first consider a scalar optimal-control problem whose optimal trajectory
crosses the nondifferentiable point \(x=0\). Since a closed-form solution is
available, this example allows direct comparison of the computed control,
state, crossing time, and objective value with their exact counterparts.

The problem is
\begin{equation}
\label{eq:constructed-problem}
\begin{aligned}
\min_{u(\cdot)}\quad
J(u)
&=
\frac{\gamma}{2}\bigl(x(2)-x_d\bigr)^2
+\frac{1}{2}\int_{0}^{2}u(t)^2\,dt,\\
\text{subject to}\quad
\dot{x}(t)
&=
|x(t)|+u(t),\\
x(0)&=-a,\\
0&\leq u(t)\leq u_{\max},
\end{aligned}
\end{equation}
where \(a>0\), \(\gamma>0\), and
\begin{equation}
\label{eq:constructed-xd}
x_d
=
a+\frac{2a}{\gamma(e^2-1)}.
\end{equation}
The target \(x_d\) is selected so that the optimal trajectory crosses
\(x=0\), where the dynamics are continuous but nondifferentiable. For the numerical experiment, we use
\begin{equation}
\label{eq:constructed-parameters}
a=0.2,
\qquad
\gamma=1,
\qquad
u_{\max}=0.2.
\end{equation}

\subsubsection{Analytical solution}

The detailed minimum-effort derivation is provided in Supporting Information,
Section S2. The expressions used to obtain the exact solution are retained here.

\begin{equation}
\label{eq:u-minus-general}
u_-(t)
=
\frac{2ae^t}{e^{2\tau}-1},
\qquad
0\leq t\leq\tau.
\end{equation}

\begin{equation}
\label{eq:u-plus-general}
u_+(t)
=
\frac{2y\,e^{2-t}}{e^{2(2-\tau)}-1},
\qquad
\tau\leq t\leq2.
\end{equation}

\begin{equation}
\label{eq:tau-y-star}
\tau^*=1,
\qquad
y^*=a.
\end{equation}

Substituting \eqref{eq:tau-y-star} into
\eqref{eq:u-minus-general} and \eqref{eq:u-plus-general} gives the exact
optimal control
\begin{equation}
\label{eq:exact-control}
u^*(t)
=
\begin{cases}
\dfrac{2ae^t}{e^2-1},
&0\leq t\leq1,\\[8pt]
\dfrac{2ae^{2-t}}{e^2-1},
&1<t\leq2.
\end{cases}
\end{equation}
Solving the corresponding linear state equations on the two sides of the
crossing and imposing continuity at \(t=1\) gives
\begin{equation}
\label{eq:exact-state}
x^*(t)
=
\frac{a\left(e^t-e^{2-t}\right)}{e^2-1},
\qquad
0\leq t\leq2.
\end{equation}
The exact trajectory satisfies
\begin{equation}
x^*(0)=-a,
\qquad
x^*(1)=0,
\qquad
x^*(2)=a.
\end{equation}

The optimal control attains its maximum at the crossing time:
\begin{equation}
u^*(1)
=
\frac{2ae}{e^2-1}.
\end{equation}
For \(a=0.2\),
\begin{equation}
u^*(1)
\approx0.1701836
<
u_{\max}=0.2.
\end{equation}
Thus, the control bound is inactive at the optimum, confirming that the
unconstrained minimum-effort construction above is admissible.

\subsubsection{Numerical comparison}

HPDPO is applied directly to the nonsmooth vector field
\begin{equation}
f(x,u)=|x|+u,
\end{equation}
without differentiating the absolute-value term. Figure
\ref{fig:constructed-exact-vs-saddle} compares the computed control and state
with the exact solutions in \eqref{eq:exact-control} and
\eqref{eq:exact-state}.

\begin{figure*}[ht]
    \centering
    \includegraphics[width=0.95\textwidth]{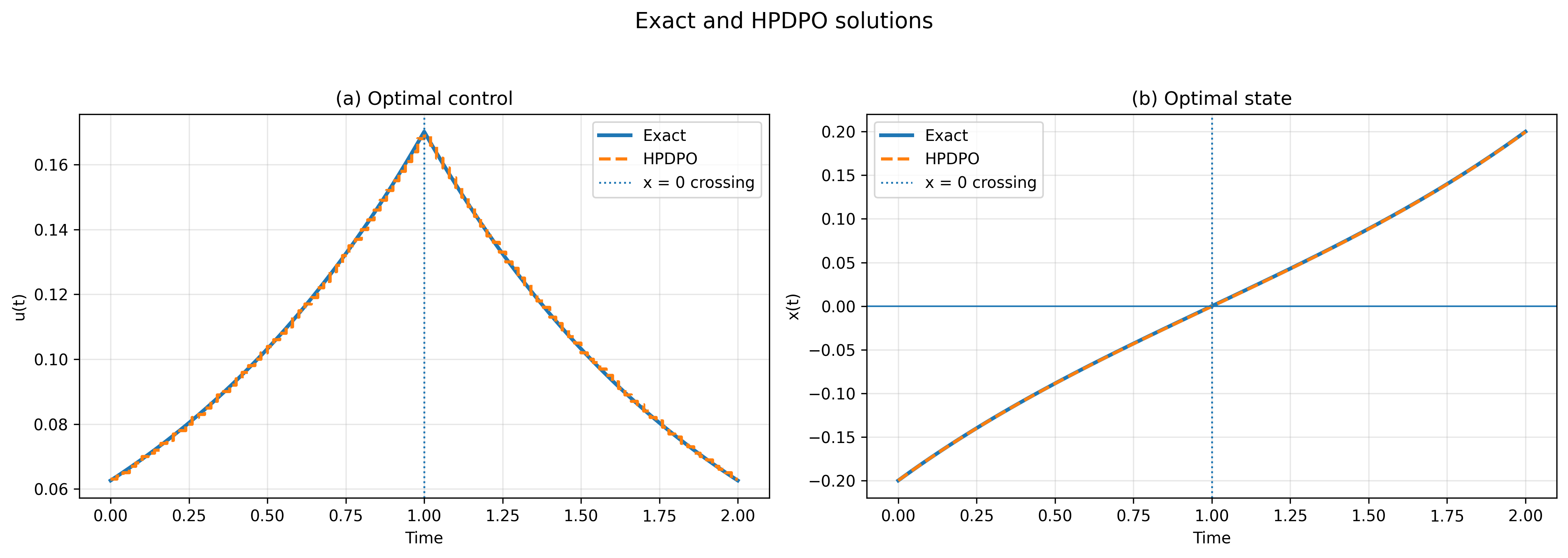}
    \caption{
    Comparison of the exact optimal control
    and state trajectories with those obtained using HPDPO for Example 1. The vertical
    dotted line marks the exact crossing time \(t=1\), at which the state
    passes through the nondifferentiable point \(x=0\).
    }
    \label{fig:constructed-exact-vs-saddle}
\end{figure*}

To quantify the numerical accuracy of the method, let
\(\widehat{u}_k\) and \(\widehat{x}_k\) denote the computed control and state,
respectively, and let \(u^*(t_k)\) and \(x^*(t_k)\) denote the exact solutions
evaluated on the same grid. The root-mean-square errors are defined by
\begin{equation}
\label{eq:constructed-rmse}
\begin{aligned}
\operatorname{RMSE}_u
&=
\left[
\frac{1}{M}
\sum_{k=0}^{M-1}
\left(
\widehat{u}_k-u^*(t_k)
\right)^2
\right]^{1/2},\\
\operatorname{RMSE}_x
&=
\left[
\frac{1}{M+1}
\sum_{k=0}^{M}
\left(
\widehat{x}_k-x^*(t_k)
\right)^2
\right]^{1/2}.
\end{aligned}
\end{equation}
These quantities measure the average discrepancies between the computed and
exact control and state trajectories over the full time horizon. We also
report the objective difference \(\lvert J_h-J_h^*\rvert\), where \(J_h^*\)
is obtained by evaluating the exact control using the same time
discretization, state solver, and quadrature rule used for the numerical
methods.

Table~\ref{tab:constructed-accuracy} compares HPDPO with two widely used
model-free reinforcement-learning methods: proximal policy optimization (PPO)
and soft actor--critic (SAC) \cite{Schulman2017,Haarnoja2018}. PPO is an
on-policy actor--critic method based on a clipped policy-update objective,
whereas SAC is an off-policy actor--critic method that uses replay data and
entropy-regularized policy updates. All three methods use the same time
horizon, time step, control bounds, RK4 transition map, and discrete
objective.

PPO and SAC were each trained for \(250{,}000\) environment transitions.
Because one episode contains
\[
M=\frac{T}{\Delta t}=\frac{2}{0.02}=100
\]
control decisions, this training budget corresponds to \(2{,}500\)
full-horizon-equivalent episodes. Because PPO and SAC involve stochastic
training, each method was evaluated over five independent random seeds.
Table~\ref{tab:constructed-accuracy} reports the mean and standard deviation
across these five runs.

\begin{table*}[ht]
\centering
\small
\caption{
Numerical comparison of HPDPO, PPO, and SAC for Example 1. PPO and SAC
results are reported as mean \(\pm\) standard deviation over five independent
random seeds.
}
\label{tab:constructed-accuracy}
\begin{tabular}{lccc}
\hline
Quantity & HPDPO & PPO & SAC \\
\hline
Time step \(\Delta t\)
& \(0.02\)
& \(0.02\)
& \(0.02\) \\
Optimization/training budget
& \(25\) outer iterations
& \(2.5\times10^{5}\) transitions
& \(2.5\times10^{5}\) transitions \\
Discrete objective \(J_h\)
& \(1.45\times10^{-2}\)
& \((1.4933\pm0.0147)\times10^{-2}\)
& \((1.5586\pm0.0428)\times10^{-2}\) \\
Control RMSE, \(\operatorname{RMSE}_u\)
& \(4.99\times10^{-4}\)
& \((1.732\pm0.201)\times10^{-2}\)
& \((3.409\pm0.723)\times10^{-2}\) \\
State RMSE, \(\operatorname{RMSE}_x\)
& \(2.12\times10^{-4}\)
& \((1.473\pm0.308)\times10^{-2}\)
& \((1.421\pm0.494)\times10^{-2}\) \\
Objective difference, \(\lvert J_h-J_h^*\rvert\)
& \(7.26\times10^{-8}\)
& \((4.52\pm1.47)\times10^{-4}\)
& \((1.10\pm0.43)\times10^{-3}\) \\
\hline
\end{tabular}
\end{table*}

The HPDPO control and state trajectories are nearly indistinguishable from
the exact solutions at the scale of
Fig.~\ref{fig:constructed-exact-vs-saddle}. Across the five independent
reinforcement-learning runs, PPO and SAC exhibited substantially larger
control and objective errors than HPDPO. Their average state errors were also
larger by approximately two orders of magnitude. PPO produced the smaller
average control error and objective gap, whereas the average state RMSEs of
PPO and SAC were similar. HPDPO therefore recovered the known optimal control
and corresponding state trajectory with substantially smaller numerical
errors on this deterministic fixed-initial-state benchmark.

\subsection{Example 2: Nonsmooth Benchmark}
\label{sec:published-nonsmooth-benchmark}

We next consider Example~15 of \cite{NooriSkandari2016}, a nonsmooth
optimal-control problem with a known exact solution. In contrast to
Example~1, the control objective in this problem is independent of the state,
while the state trajectory passes through a nondifferentiable point. The
example therefore tests whether the proposed method can recover the optimal
control and propagate the resulting trajectory through the nonsmooth
dynamics.

The problem is
\begin{equation}
\label{eq:example15-problem}
\begin{aligned}
\min_{u(\cdot)}\quad
J(u)
&=
\int_{0}^{2}u(t)\bigl(u(t)-t\bigr)\,dt,\\
\text{subject to}\quad
\dot{x}(t)
&=
-\left|x(t)-\frac{\mu}{2}\right|
+\left|\frac{\mu}{2}\right|
+\frac{2\bigl(u(t)+1\bigr)}{t+2}
-1,\\
x(0)&=0.1,\\
0&\leq x(t)\leq1,\\
-1&\leq u(t)\leq1,
\qquad 0\leq t\leq2.
\end{aligned}
\end{equation}

Following \cite{NooriSkandari2016}, we set
\begin{equation}
\mu=1.
\end{equation}
The state equation then becomes
\begin{equation}
\label{eq:example15-dynamics-mu1}
\dot{x}(t)
=
-\left|x(t)-0.5\right|
+0.5
+\frac{2\bigl(u(t)+1\bigr)}{t+2}
-1,
\end{equation}
which is continuous but nondifferentiable at \(x=0.5\). The exact solution reported in
\cite{NooriSkandari2016} is
\begin{equation}
\label{eq:example15-exact-control}
u^*(t)=\frac{t}{2},
\qquad
0\leq t\leq2,
\end{equation}
with optimal objective value
\begin{equation}
\label{eq:example15-exact-objective}
J^*=-\frac{2}{3}.
\end{equation}

The corresponding exact state trajectory is
\begin{equation}
\label{eq:example15-exact-state}
x^*(t)
=
\begin{cases}
0.1e^t,
&0\leq t\leq\ln(5),\\[6pt]
1-2.5e^{-t},
&\ln(5)<t\leq2.
\end{cases}
\end{equation}
The trajectory reaches the nondifferentiable state value \(x=0.5\) at
\begin{equation}
\label{eq:example15-crossing-time}
t_s=\ln(5),
\end{equation}
since
\begin{equation}
x^*\bigl(\ln(5)\bigr)=0.5.
\end{equation}

\subsubsection{Numerical comparison}

HPDPO is applied directly to
\eqref{eq:example15-dynamics-mu1}, without differentiating the nonsmooth term
\(\lvert x-0.5\rvert\). The interval \([0,2]\) is discretized using
\(\Delta t=0.01\), and the control candidates are selected from uniformly
spaced values in \([-1,1]\). Figure~\ref{fig:published-exact-vs-saddle}
compares the HPDPO control and state trajectories with the exact solutions in
\eqref{eq:example15-exact-control} and
\eqref{eq:example15-exact-state}.
\begin{table*}[ht]
\centering
\small
\caption{
Numerical comparison of HPDPO, PPO, and SAC for Example 2. PPO and SAC
results are reported as mean \(\pm\) standard deviation over five independent
random seeds.
}
\label{tab:published-benchmark-accuracy}
\begin{tabular}{lccc}
\hline
Quantity & HPDPO & PPO & SAC \\
\hline
Time step \(\Delta t\)
& \(0.01\)
& \(0.01\)
& \(0.01\) \\

Optimization/training budget
& \(8\) outer iterations
& \(2.5\times10^{5}\) transitions
& \(2.5\times10^{5}\) transitions \\

Discrete objective \(J_h\)
& \(-0.661675\)
& \(-0.2434 \pm 0.1775\)
& \(-0.4629 \pm 0.1342\) \\

Control RMSE, \(\operatorname{RMSE}_u\)
& \(7.30\times10^{-17}\)
& \(0.4480 \pm 0.1027\)
& \(0.2945 \pm 0.1256\) \\

State RMSE, \(\operatorname{RMSE}_x\)
& \(4.56\times10^{-3}\)
& \(0.4343 \pm 0.0978\)
& \(0.2463 \pm 0.0622\) \\

Discrete objective gap,
\(\lvert J_h-J_h^*\rvert\)
& \(<10^{-14}\)
& \(0.4183 \pm 0.1775\)
& \(0.1987 \pm 0.1342\) \\

Feasible runs
& \(1/1\)
& \(2/5\)
& \(5/5\) \\
\hline
\end{tabular}
\end{table*}
\begin{figure*}[ht]
    \centering
    \includegraphics[width=0.95\textwidth]{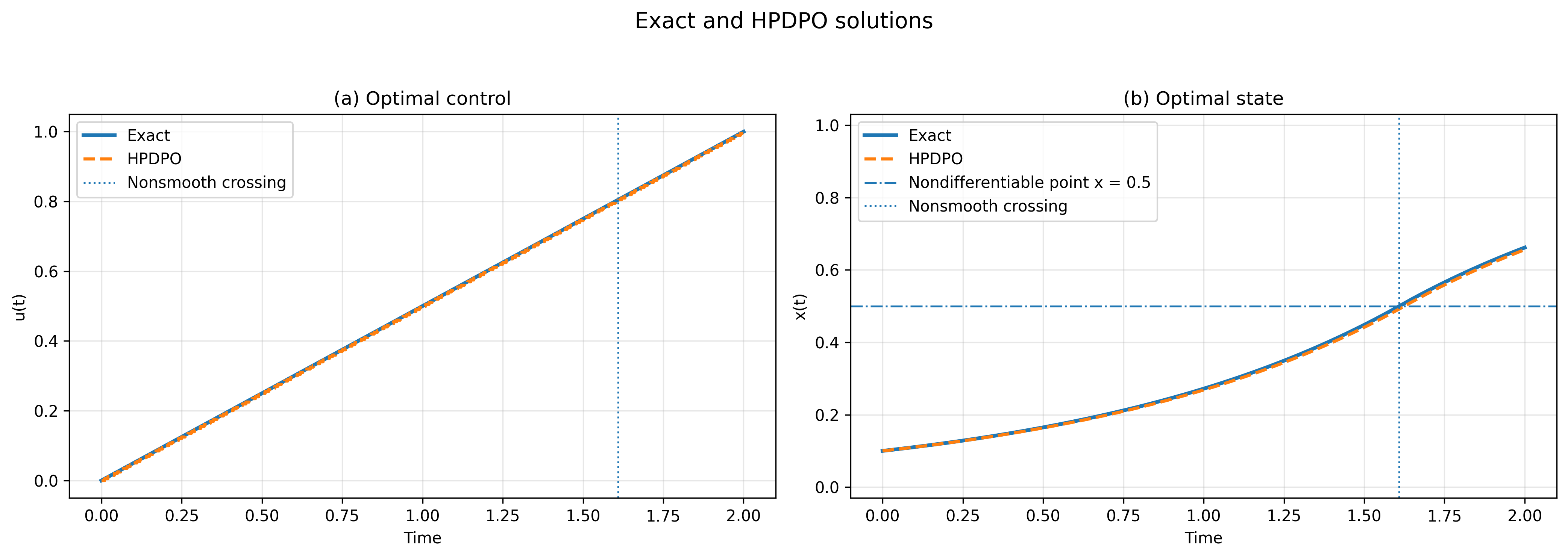}
    \caption{
    Comparison of the exact optimal control and state trajectories with those
    obtained using HPDPO for Example 2. The horizontal dash-dotted line marks the
    nondifferentiable state value \(x=0.5\), and the vertical dotted line marks
    the exact crossing time \(t=\ln(5)\).
    }
    \label{fig:published-exact-vs-saddle}
\end{figure*}

The control and state errors are evaluated using the root-mean-square
definitions in \eqref{eq:constructed-rmse}. The exact control is evaluated at
the left endpoints of the control intervals, consistently with the discrete
policy representation. The discrete reference objective \(J_h^*\) is obtained
by evaluating the exact control using the same time grid, RK4 state solver, and
quadrature rule as the numerical methods. For \(\Delta t=0.01\), this gives
\[
J_h^*=-0.661675.
\]

For comparison, the same problem was solved using PPO and SAC. Each
reinforcement-learning method was trained for \(250{,}000\) environment
transitions. Since each full-horizon episode contains
\[
M=\frac{T}{\Delta t}=\frac{2}{0.01}=200
\]
control decisions, this training budget corresponds to \(1{,}250\)
full-horizon-equivalent episodes. Because PPO and SAC are stochastic, each
method was evaluated over five independent random seeds. Table
\ref{tab:published-benchmark-accuracy} reports the mean and standard deviation
across these five runs. We also report the number of feasible runs, since not
all PPO runs satisfied the state constraint.

HPDPO accurately recovers the known optimum while operating directly on the
nonsmooth dynamics. Its control agrees with the exact discrete control to
numerical precision, and its discrete objective gap is below \(10^{-14}\).
Across the five reinforcement-learning runs, SAC produced smaller average
control, state, and objective errors than PPO, and all SAC runs satisfied the
state constraint. In contrast, only two of the five PPO runs were feasible.
Nevertheless, both reinforcement-learning methods produced substantially
larger errors than HPDPO. These results show that HPDPO provides a much more
accurate solution for this deterministic fixed-initial-state nonsmooth
benchmark.
\subsection{Example 3: Nonlinear Cascade with Analytical and Learned Dynamics}
\label{sec:cascade-example}

The third study considers a four-state nonlinear cascade model motivated by
previous computational formulations of the Alzheimer's disease biomarker
cascade \cite{Petrella2019,Petrella2024}. We first compare HPDPO with the
forward--backward sweep (FBS) method using the known analytical dynamics.
We then replace the untreated right-hand side by a learned ReLU neural network
and apply HPDPO directly to the resulting nonsmooth model.

\subsubsection{Analytical cascade model}

Let
\begin{equation}
x(t)
=
\begin{pmatrix}
A(t)\\
\tau_p(t)\\
N(t)\\
C(t)
\end{pmatrix}
\in\mathbb{R}^4.
\end{equation}
The controlled dynamics are
\begin{equation}
\label{eq:cascade-dynamics}
\begin{aligned}
\dot{A}
&=
\lambda_A A(K_A-A)-uA,\\
\dot{\tau}_p
&=
\lambda_{\tau_p A}A
+
\lambda_{\tau_p}\tau_p(K_{\tau_p}-\tau_p),\\
\dot{N}
&=
\lambda_{N\tau_p}\tau_p
+
\lambda_N N(K_N-N),\\
\dot{C}
&=
\lambda_{CN}rN
+
\lambda_C C(K_C-C).
\end{aligned}
\end{equation}
The control acts directly only on \(A\), while its effects on
\(\tau_p\), \(N\), and \(C\) are transmitted through the cascade. The initial
condition is
\begin{equation}
\label{eq:cascade-initial-condition}
x(0)
=
\begin{pmatrix}
0.05\\
0\\
0\\
0
\end{pmatrix},
\end{equation}
and the parameter vector is
\begin{equation}
\label{eq:cascade-parameters}
\begin{aligned}
\theta
=
(&\lambda_A,K_A,
\lambda_{\tau_p A},\lambda_{\tau_p},K_{\tau_p},
\lambda_{N\tau_p},\lambda_N,K_N,
\lambda_{CN},r,\lambda_C,K_C)\\
=
&(0.08,1.0,
0.025,0.05,1.0,
0.025,0.05,1.0,
0.001,25.0,0.05,1.0).
\end{aligned}
\end{equation}

We consider the finite-horizon optimal-control problem
\begin{equation}
\label{eq:cascade-ocp}
\begin{aligned}
\min_{u(\cdot)}
\quad
J(u)
&=
\int_0^T
\left[
C(t)+\alpha u(t)^2
\right]\,dt,\\
\text{subject to}\quad
\dot{x}(t)
&=
f_{\mathrm{true}}(x(t),u(t)),\\
x(0)&=x_0,\\
0&\leq u(t)\leq0.2,
\end{aligned}
\end{equation}
where
\begin{equation}
T=200,
\qquad
\alpha=0.1.
\end{equation}
The objective balances cumulative cognitive impairment, represented by
\(C(t)\), against treatment effort. No terminal cost is included. Because the
analytical dynamics are known in this example, they provide both an FBS
reference solution and supervised targets for constructing a learned
right-hand-side model.

\subsubsection{Learned ReLU approximation}
\label{sec:cascade-learned-dynamics}

To separate the unknown autonomous dynamics from the known treatment
mechanism, we write
\begin{equation}
f_{\mathrm{true}}(x,u)
=
f_{\mathrm{true}}(x,0)+b(x,u),
\end{equation}
where
\begin{equation}
b(x,u)
=
\begin{pmatrix}
-uA\\
0\\
0\\
0
\end{pmatrix}.
\end{equation}
A feedforward ReLU network
\(g_\phi:\mathbb{R}^4\rightarrow\mathbb{R}^4\)
is trained to approximate only the untreated right-hand side,
\begin{equation}
\label{eq:nn-rhs-target}
g_\phi(x)
\approx
f_{\mathrm{true}}(x,0).
\end{equation}
The corresponding controlled learned model is therefore
\begin{equation}
\label{eq:hybrid-neural-control}
\dot{x}
=
g_\phi(x)
+
\begin{pmatrix}
-uA\\
0\\
0\\
0
\end{pmatrix}.
\end{equation}
Thus, the autonomous dynamics are learned, while the known treatment
mechanism remains explicit.

Training states are sampled independently from the feasible region
\begin{equation}
\label{eq:feasible-box}
\Omega
=
[0,1]\times[0,1.5]\times[0,1.5]\times[0,1.5].
\end{equation}
For samples \(x_i\sim\operatorname{Uniform}(\Omega)\), the corresponding
targets are
\begin{equation}
y_i=f_{\mathrm{true}}(x_i,0),
\end{equation}
and the network parameters are obtained by minimizing
\begin{equation}
\label{eq:rhs-training-objective}
\min_{\phi}
\frac{1}{N}
\sum_{i=1}^{N}
\left\|
g_\phi(x_i)-y_i
\right\|^2.
\end{equation}

Before using the learned model in the control problem, we validate
\(g_\phi\) at both the derivative and trajectory levels.
Figure~\ref{fig:learned-dynamics-validation}(a) compares the componentwise
predictions of the learned right-hand side with the corresponding analytical
untreated right-hand side along a reference trajectory. The agreement for
\(A\), \(\tau_p\), \(N\), and \(C\) shows that the ReLU network accurately
reproduces the local vector field along the region visited by the system.
Figure~\ref{fig:learned-dynamics-validation}(b) provides a second validation
by integrating the learned dynamics and comparing the resulting state
trajectories with those of the analytical model under the same treatment.
The close trajectory agreement indicates that the right-hand-side
approximation errors do not accumulate substantially over the simulated time
horizon.

\begin{figure*}[ht]
    \centering
    \begin{minipage}[t]{0.48\textwidth}
        \centering
        \includegraphics[
            height=0.4\textheight,
            width=\textwidth,
            keepaspectratio
        ]{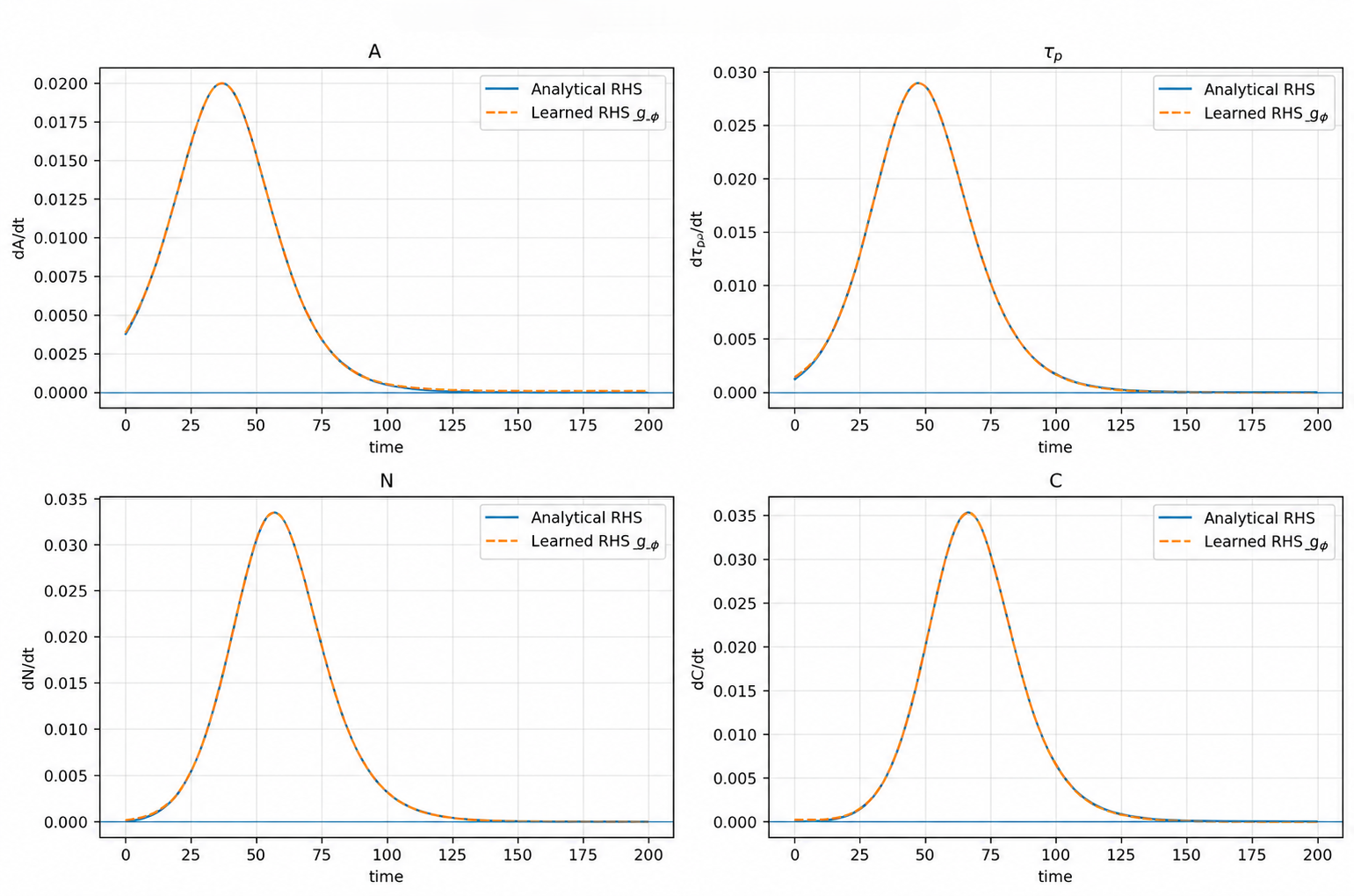}\\[-1mm]
        \small (a) Right-hand-side validation.
    \end{minipage}
    \hfill
    \begin{minipage}[t]{0.48\textwidth}
        \centering
        \includegraphics[
            height=0.4\textheight,
            width=\textwidth,
            keepaspectratio
        ]{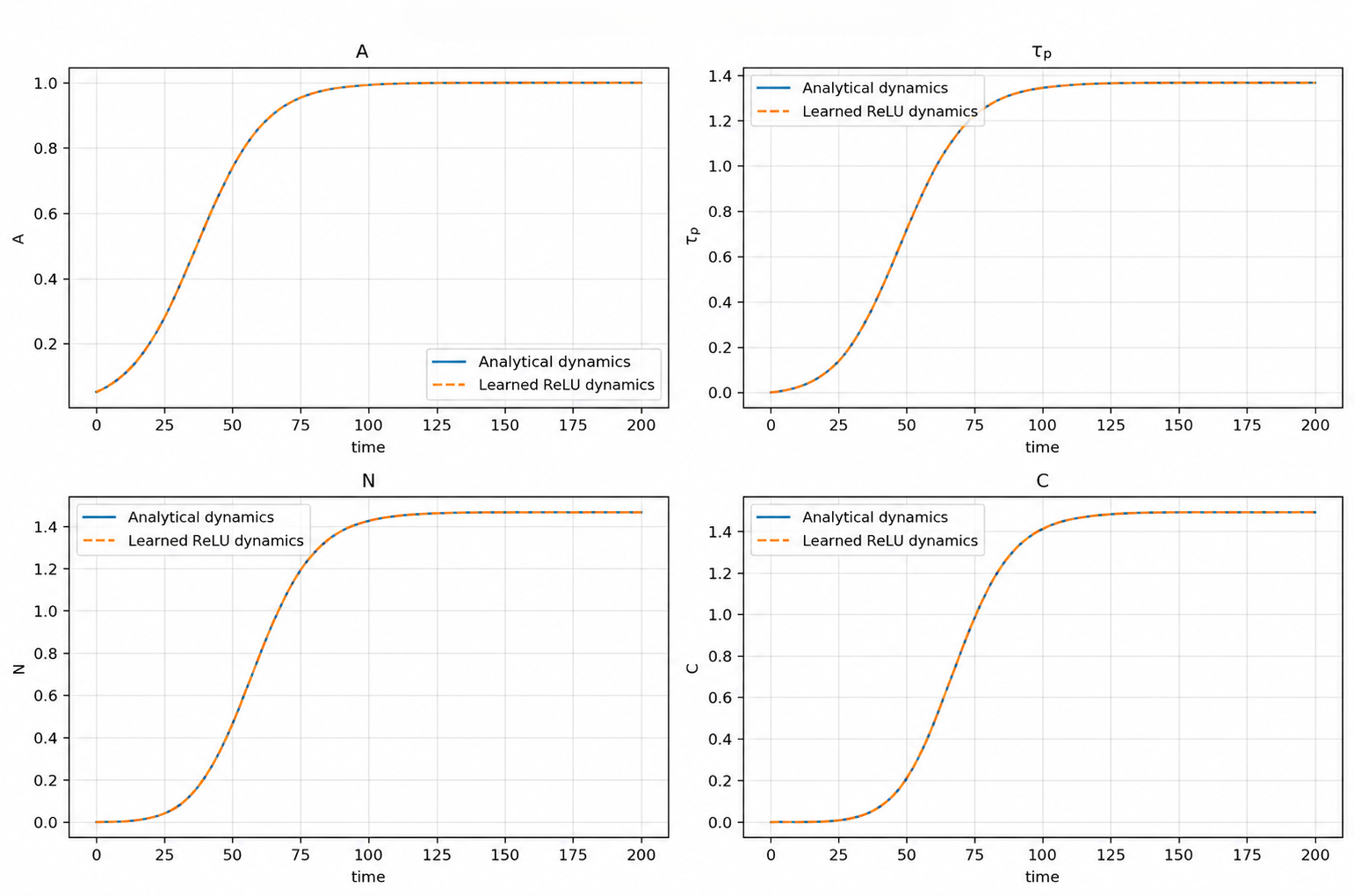}\\[-1mm]
        \small (b) Trajectory validation.
    \end{minipage}
    \caption{
    Validation of the learned right-hand side \(g_\phi\) for Example 3.
    Panel (a) compares the learned right-hand side with the analytical
    untreated vector field \(f_{\mathrm{true}}(x,0)\) along a reference
    trajectory for \(A\), \(\tau_p\), \(N\), and \(C\).
    Panel (b) compares the corresponding state trajectories generated by the
    analytical and learned ReLU dynamics under the same treatment.
    Agreement at both the derivative and trajectory levels validates the
    learned dynamics before their use within HPDPO.
    }
    \label{fig:learned-dynamics-validation}
\end{figure*}

\subsubsection{Optimal-control comparison}

We solve the optimal-control problem in \eqref{eq:cascade-ocp} using three
model--optimizer combinations:
\begin{enumerate}
    \item FBS with the analytical dynamics;
    \item HPDPO with the analytical dynamics;
    \item HPDPO with the learned ReLU dynamics.
\end{enumerate}

Figure~\ref{fig:cascade-control-comparison} compares the resulting treatment
policies. All three policies have the same qualitative structure: treatment
is strongest during the early portion of the horizon and then decreases
gradually toward zero. The agreement between the two HPDPO policies also
shows that replacing the analytical right-hand side by the learned ReLU model
does not substantially alter the resulting control strategy.

\begin{figure*}[ht]
    \centering
    \includegraphics[width=0.6\textwidth]
    {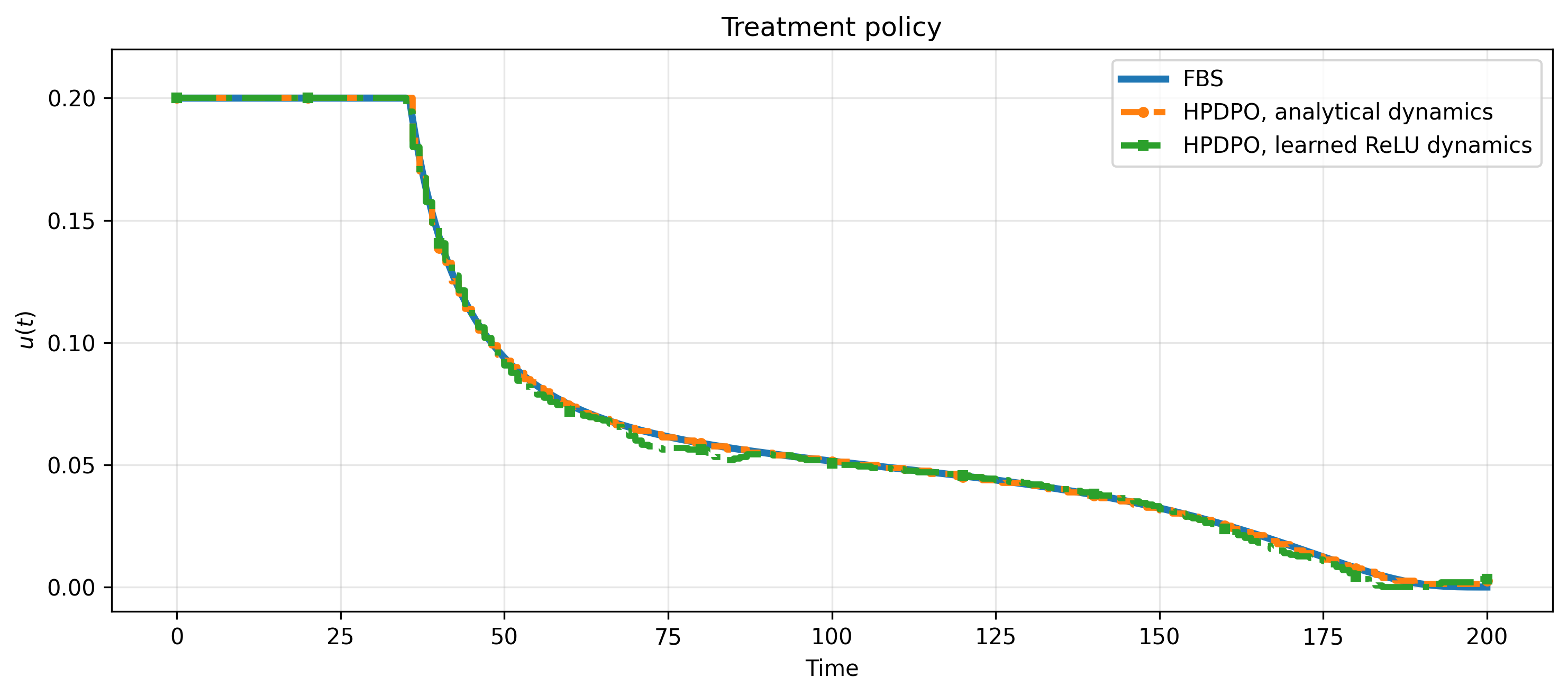}
    \caption{
    Comparison of the optimized treatment policies obtained using FBS with
    the analytical dynamics, HPDPO with the analytical dynamics, and HPDPO
    with the learned ReLU dynamics, Example 3.
    }
    \label{fig:cascade-control-comparison}
\end{figure*}

Figure~\ref{fig:cascade-state-comparison} compares the corresponding state
trajectories. For each method, the displayed trajectory is generated using
its own optimized treatment policy and its corresponding dynamical model.
The HPDPO trajectory based on the analytical model is therefore generated
using \(f_{\mathrm{true}}\), whereas the learned-dynamics HPDPO trajectory is
generated using \(g_\phi\) together with the explicit treatment term in
\eqref{eq:hybrid-neural-control}.

\begin{figure*}[ht]
    \centering
    \begin{minipage}[t]{0.48\textwidth}
        \centering
        \includegraphics[width=\textwidth]
        {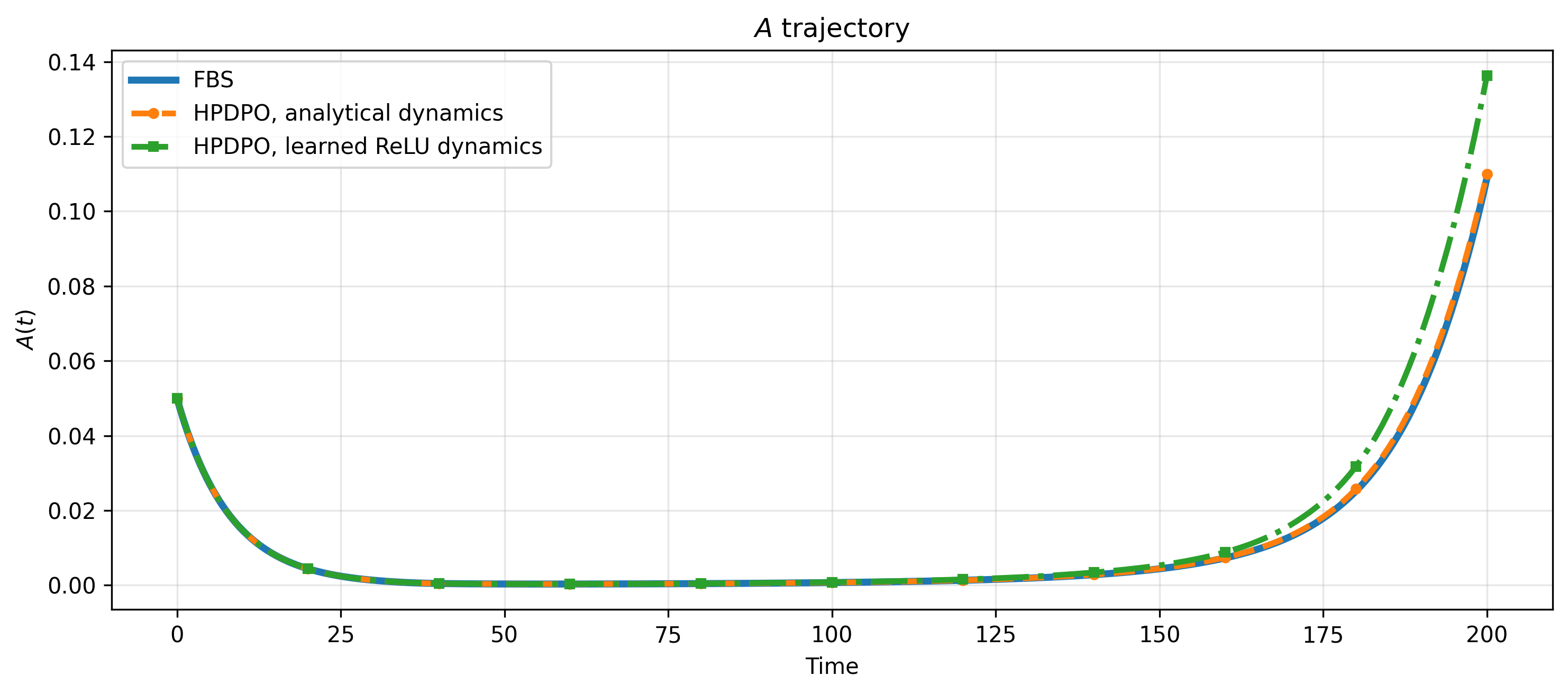}\\[-1mm]
        \small (a) \(A(t)\)
    \end{minipage}
    \hfill
    \begin{minipage}[t]{0.48\textwidth}
        \centering
        \includegraphics[width=\textwidth]
        {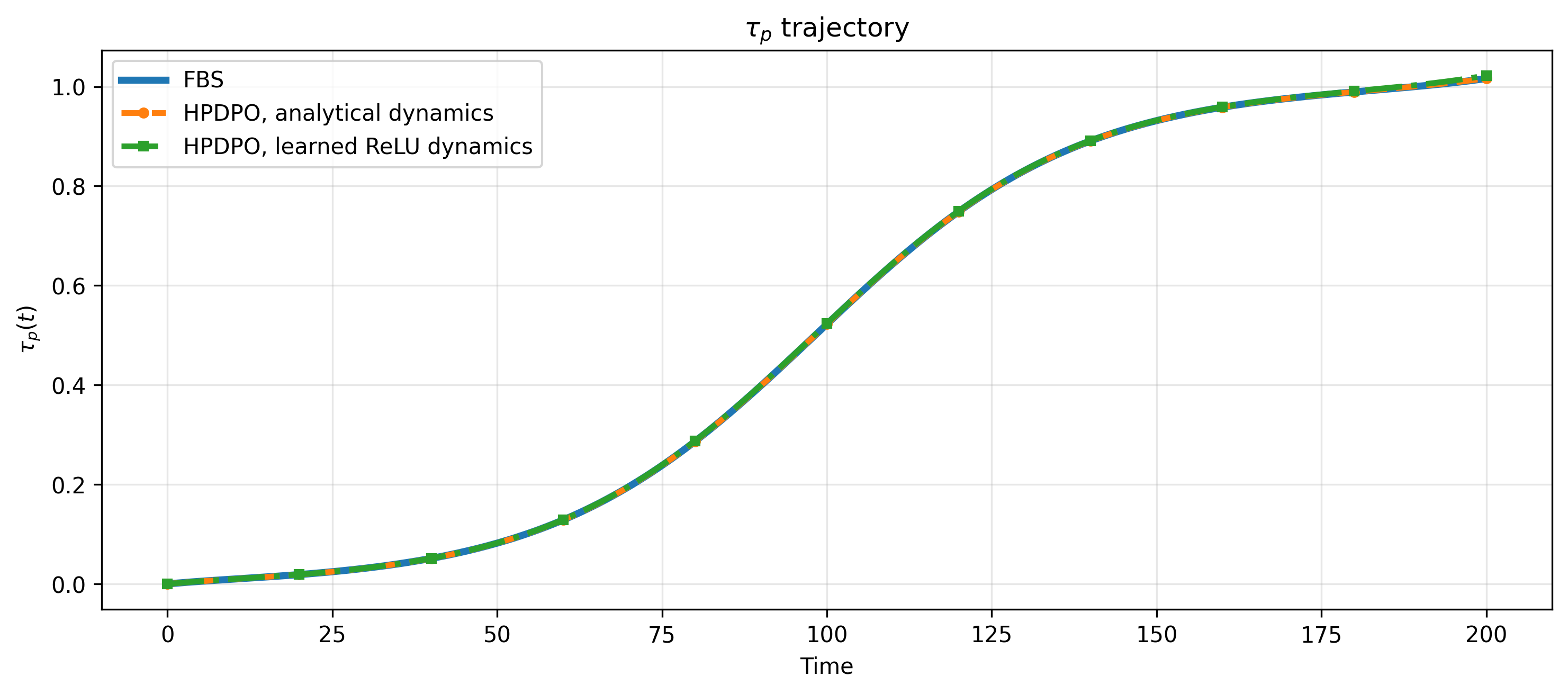}\\[-1mm]
        \small (b) \(\tau_p(t)\)
    \end{minipage}

    \vspace{0.5em}

    \begin{minipage}[t]{0.48\textwidth}
        \centering
        \includegraphics[width=\textwidth]
        {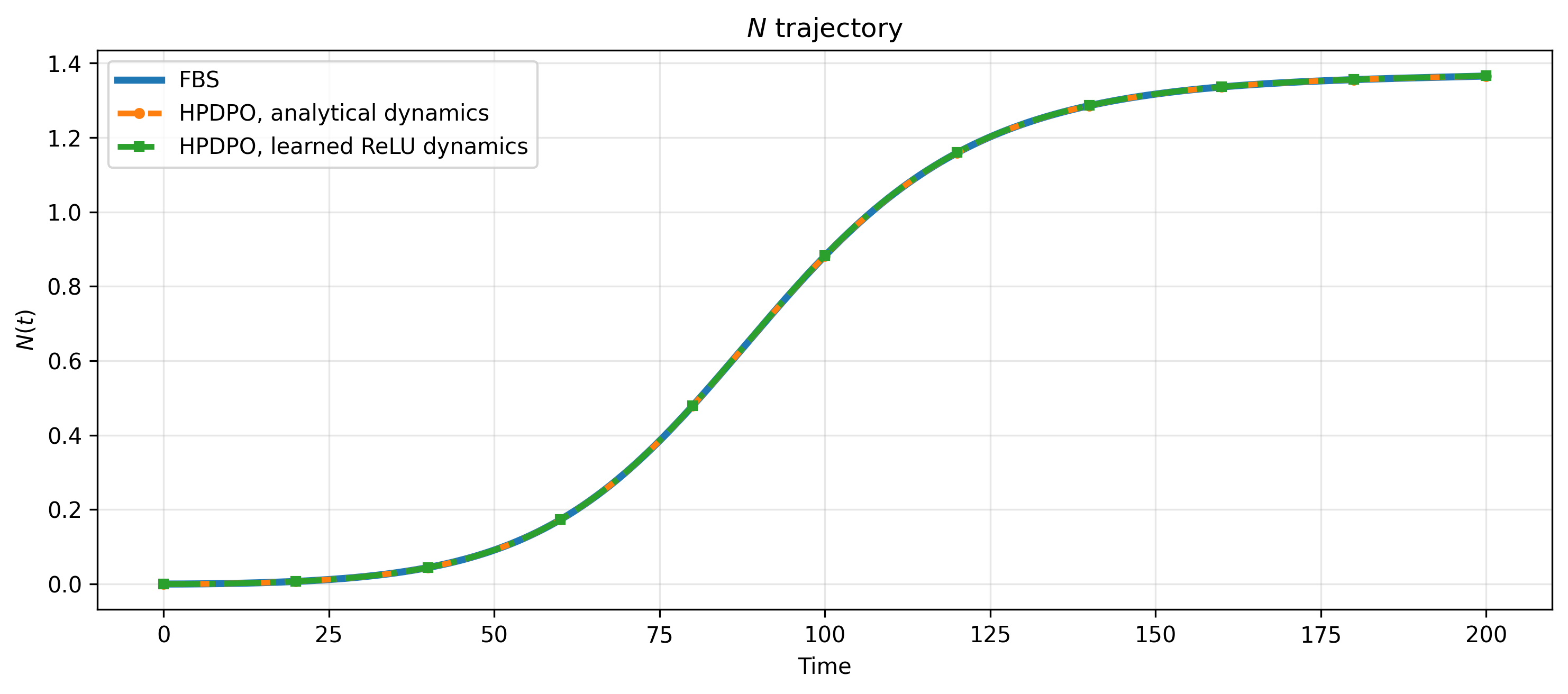}\\[-1mm]
        \small (c) \(N(t)\)
    \end{minipage}
    \hfill
    \begin{minipage}[t]{0.48\textwidth}
        \centering
        \includegraphics[width=\textwidth]
        {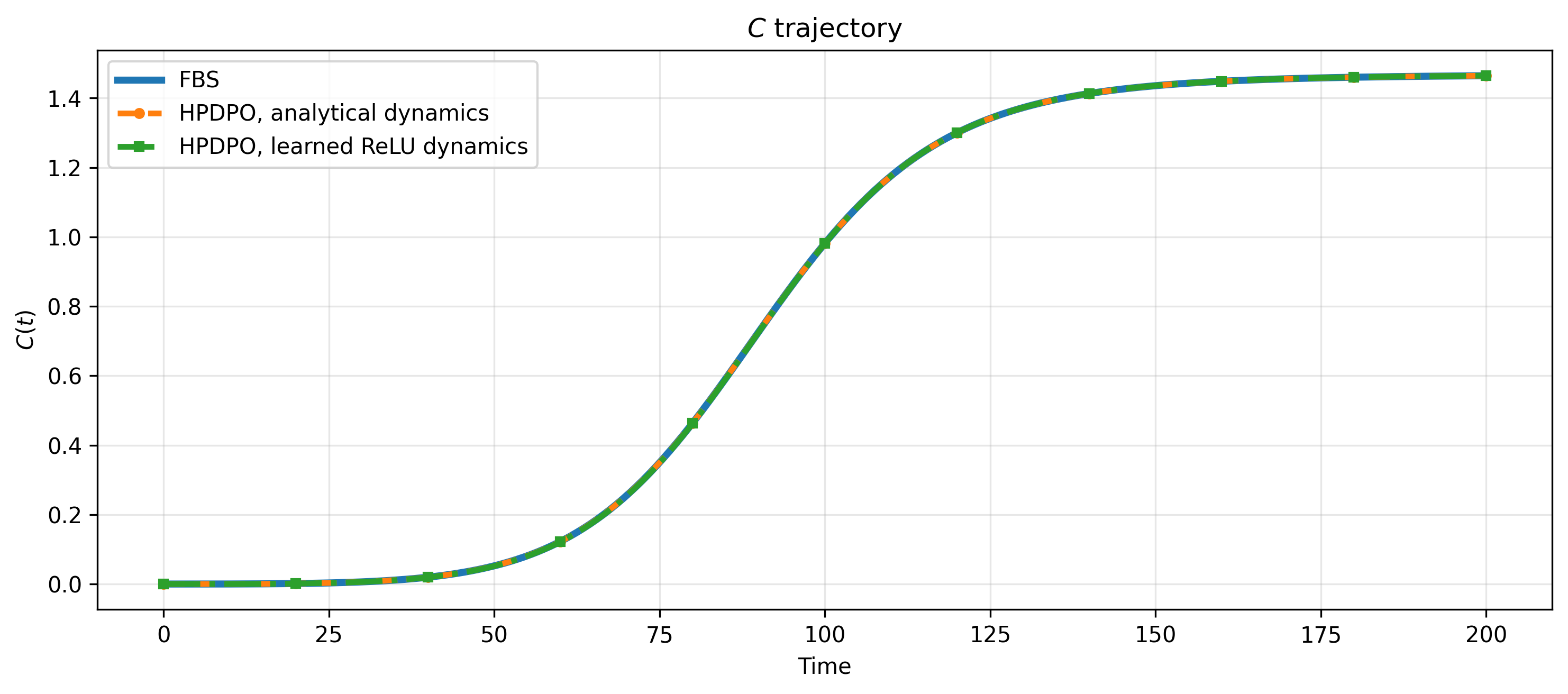}\\[-1mm]
        \small (d) \(C(t)\)
    \end{minipage}

    \caption{
    State trajectories obtained using FBS with the analytical dynamics,
    HPDPO with the analytical dynamics, and HPDPO with the learned ReLU
    dynamics, Example 3.
    }
    \label{fig:cascade-state-comparison}
\end{figure*}

For a common-model quantitative comparison, the policy obtained by HPDPO
using the learned ReLU dynamics is additionally evaluated under the
analytical dynamics. Table~\ref{tab:cascade-method-comparison} reports the
resulting objective error and the control and state RMSE values relative to
the FBS reference. All RMSE values are evaluated at common time points.

\begin{table*}[ht]
\centering
\caption{
Comparison of HPDPO using the analytical and learned ReLU dynamics with the
FBS reference for Example 3. For the learned-dynamics case, the HPDPO policy
is evaluated under the analytical dynamics when computing the reported state
and objective errors.
}
\label{tab:cascade-method-comparison}
\begin{tabular}{lcc}
\hline
Quantity
& HPDPO, analytical dynamics
& HPDPO, learned ReLU dynamics \\
\hline

Absolute objective error
& \(2.02\times10^{-3}\)
& \(1.83\times10^{-3}\) \\

Control RMSE
& \(1.28\times10^{-3}\)
& \(2.47\times10^{-3}\) \\

\(A\)-state RMSE
& \(3.64\times10^{-4}\)
& \(3.95\times10^{-3}\) \\

\(\tau_p\)-state RMSE
& \(8.75\times10^{-5}\)
& \(8.05\times10^{-4}\) \\

\(N\)-state RMSE
& \(1.46\times10^{-5}\)
& \(1.26\times10^{-4}\) \\

\(C\)-state RMSE
& \(2.24\times10^{-6}\)
& \(1.93\times10^{-5}\) \\
\hline
\end{tabular}
\end{table*}

Both HPDPO solutions closely agree with the FBS reference. Using the
analytical dynamics, HPDPO produces very small control, state, and objective
errors. Replacing the analytical right-hand side by the learned ReLU model
increases the state and control errors modestly, as expected from the
additional approximation step, but the resulting policy remains close to the
analytical-model solution. These results show that HPDPO can operate directly
with a learned nonsmooth dynamical model while retaining close agreement with
a derivative-based reference constructed from the known analytical system.

\subsection{Example 4: Data-Driven Anti-Amyloid Optimal Control}
\label{sec:data-driven-amyloid-control}

The final study applies HPDPO to a population-level amyloid Neural ODE
learned from longitudinal ADNI observations \cite{Weiner2010}.
The ADNI study protocols were approved by the institutional review boards
of the participating institutions, and written informed consent was obtained
from all participants or their legally authorized representatives.
Unlike the previous benchmark examples, no exact optimal-control solution
is available.
We therefore construct a differentiable approximation of the learned
right-hand side and use forward--backward sweep (FBS) on this smooth model as
a numerical reference.

Let
\begin{equation}
A(s)\in[0,1]
\end{equation}
denote normalized amyloid pathology, with larger values indicating greater
amyloid burden. The variable \(s\) represents the Alzheimer's disease
progression score (DPS) \cite{Jedynak2012,MehdipourGhazi2021}. For subject
\(i\) at visit \(j\),
\begin{equation}
\label{eq:amyloid-abc}
s_{ij}
=
\beta_i+\alpha_i t_{ij},
\end{equation}
where \(t_{ij}\) is the visit time and \(\alpha_i\) and \(\beta_i\) are fixed
subject-specific progression parameters.

\subsubsection{Learning the untreated amyloid dynamics}

The untreated population dynamics are modeled by
\begin{equation}
\label{eq:amyloid-neural-ode}
\frac{dA}{ds}
=
g_{\phi}(A),
\end{equation}
where
\begin{equation}
\label{eq:amyloid-neural-rhs}
g_{\phi}(A)
=
\operatorname{Softplus}\!\left(h_{\phi}(A)\right)A(1-A).
\end{equation}
Here, \(h_{\phi}\) is a feedforward neural network with two hidden layers,
64 units per layer, and ReLU activation functions. The ReLU activations make
the learned right-hand side continuous but generally nondifferentiable with
respect to \(A\). The Softplus transformation ensures a nonnegative untreated
progression rate, while the factor \(A(1-A)\) enforces zero untreated growth
at the boundaries of the normalized state interval.

A shared population anchor
\begin{equation}
\label{eq:amyloid-anchor}
A(0)=A_0
\end{equation}
is learned jointly with the network parameters. Learning \(A_0\) is useful
because the observations do not generally contain a measurement exactly at
\(s=0\). The anchor is initialized using training observations closest to
\(s=0\), and a weak penalty limits its deviation from this data-based
initial estimate.

For each observation, the Neural ODE is integrated from the shared anchor to
the corresponding disease-progression coordinate \(s_{ij}\). The principal
data-fitting term is the subject-balanced loss
\begin{equation}
\label{eq:amyloid-data-loss}
\mathcal{L}_{\mathrm{data}}
=
\frac{1}{N_{\mathrm{subj}}}
\sum_{i=1}^{N_{\mathrm{subj}}}
\frac{1}{n_i}
\sum_{j=1}^{n_i}
\left[
\widehat{A}(s_{ij})
-
A_{ij}^{\mathrm{obs}}
\right]^2.
\end{equation}
This weighting prevents subjects with a larger number of visits from
dominating the population fit.

The principal training objective is
\begin{equation}
\label{eq:amyloid-training-loss}
\mathcal{L}_{\mathrm{train}}
=
\mathcal{L}_{\mathrm{data}}
+
\lambda_0
\left(A_0-\widetilde{A}_0\right)^2,
\end{equation}
where \(\widetilde{A}_0\) is the initial data-based anchor estimate.
The implementation additionally includes a weak \(L_2\) regularization term
on the learned progression rate. The complete architecture, optimization
settings, and regularization details are reported in Supporting Information,
Section S5.

The data are divided at the subject level into training, validation, and
test subsets. After preprocessing, the data set contains 315 eligible
subjects and 876 amyloid observations. The learned population anchor is
\begin{equation}
\label{eq:learned-amyloid-anchor}
A_0=0.3694.
\end{equation}

Figure~\ref{fig:amyloid-population-fit} shows the observed normalized amyloid
values together with the learned population trajectory. This figure is
intended to illustrate the population-level disease-progression trend rather
than individualized subject predictions.

\begin{figure*}[ht]
    \centering
    \includegraphics[width=0.6\textwidth]
    {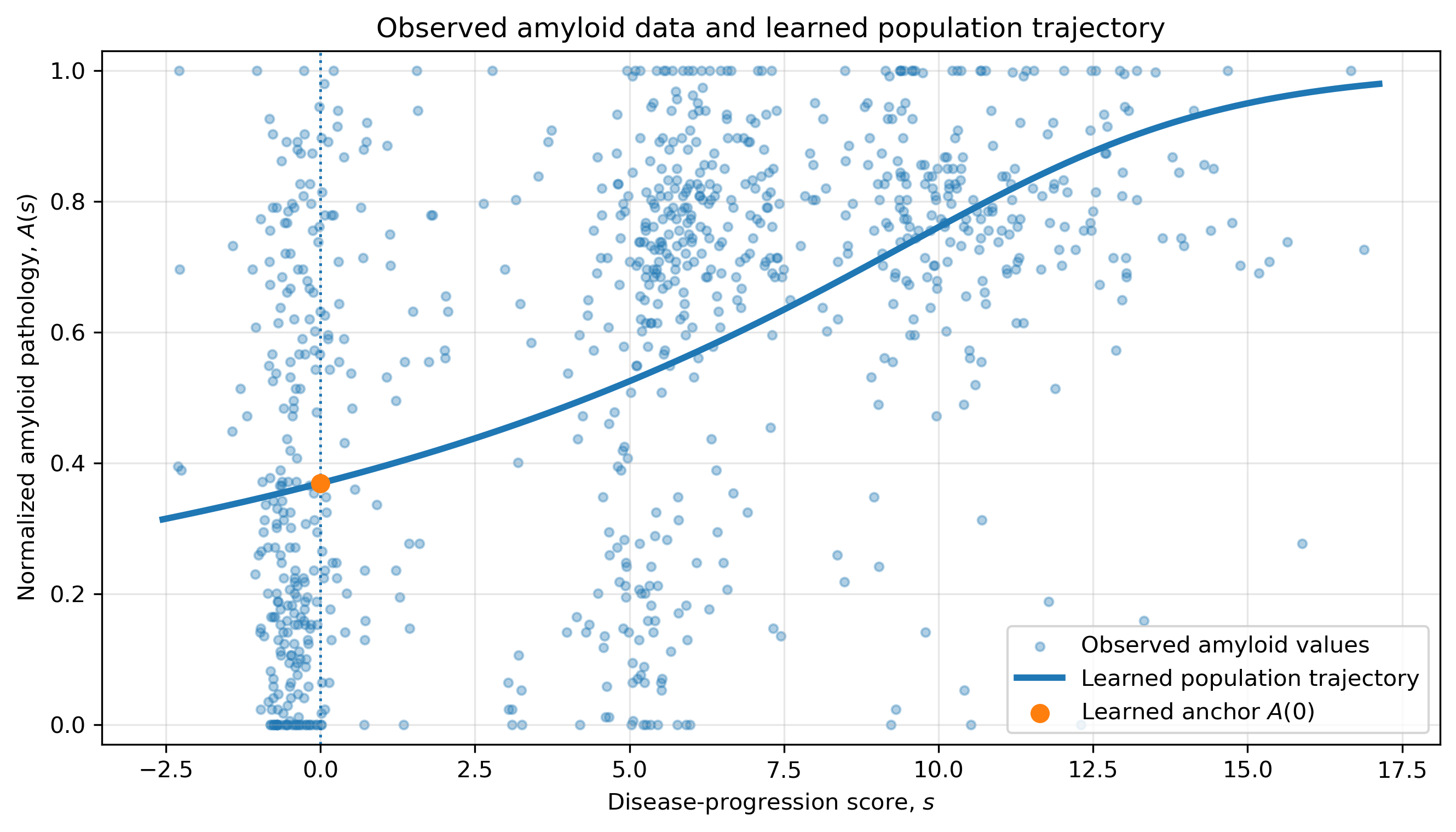}
    \caption{
    Observed normalized amyloid values and the learned population Neural ODE
    trajectory as functions of DPS. The orange point marks the learned
    population anchor \(A(0)\).
    }
    \label{fig:amyloid-population-fit}
\end{figure*}

\subsubsection{Controlled neural model}

Treatment is introduced through a state-dependent amyloid-clearance term,
giving
\begin{equation}
\label{eq:controlled-nn-amyloid}
\frac{dA}{ds}
=
g_{\phi}(A)-\kappa u(s)A(s),
\qquad
A(0)=A_0.
\end{equation}
The normalized treatment intensity satisfies
\begin{equation}
\label{eq:amyloid-control-bounds}
0\leq u(s)\leq u_{\max},
\end{equation}
where \(u=0\) represents no treatment and \(u=u_{\max}\) represents the
maximum allowed treatment intensity. The parameter \(\kappa\) controls the
strength of treatment-induced amyloid clearance.

The corresponding optimal-control problem is
\begin{equation}
\label{eq:amyloid-objective}
\begin{aligned}
\min_{u(\cdot)}
\quad
J(u)
&=
\int_{s_0}^{s_f}
\left[
A(s)+\rho u(s)^2
\right]\,ds,\\
\text{subject to}\quad
\frac{dA}{ds}
&=
g_{\phi}(A)-\kappa uA,\\
A(s_0)&=A_0,\\
0&\leq u(s)\leq u_{\max}.
\end{aligned}
\end{equation}
The objective balances cumulative amyloid burden against treatment effort.

The numerical settings are
\begin{equation}
\label{eq:amyloid-numerical-settings}
\begin{aligned}
s_0&=0,\qquad s_f=10,\qquad \Delta s=0.1,\\
\kappa&=1,\qquad \rho=0.1,\qquad u_{\max}=1.
\end{aligned}
\end{equation}
The treatment horizon is chosen within the range of DPS values represented
in the training data. The treatment parameters are prescribed for this
numerical experiment and are not estimated from the untreated ADNI
observations.

HPDPO is applied directly to the controlled Neural ODE in
\eqref{eq:controlled-nn-amyloid}. In particular, the method requires only
forward evaluations of \(g_{\phi}(A)\) and does not require derivatives or
subgradients of the ReLU network.

\subsubsection{Smooth Legendre approximation and FBS reference}

Because FBS requires differentiable dynamics, we construct a smooth
approximation of the learned untreated right-hand side over \(A\in[0,1]\):
\begin{equation}
\label{eq:legendre-approximation}
g_{\phi}(A)
\approx
g_L(A)
=
\sum_{m=0}^{p}
c_mP_m(2A-1),
\end{equation}
where \(P_m\) denotes the degree-\(m\) Legendre polynomial. A degree-\(9\)
expansion is used.

The corresponding smooth controlled model is
\begin{equation}
\label{eq:controlled-legendre-amyloid}
\frac{dA}{ds}
=
g_L(A)-\kappa u(s)A(s),
\qquad
A(0)=A_0.
\end{equation}
The initial condition, objective, treatment parameters, and control bounds
are kept unchanged. This differentiable approximation is then solved using
FBS.

Figure~\ref{fig:amyloid-rhs-approximation} compares the learned neural
right-hand side with its Legendre approximation. Their close agreement
provides the basis for using the smooth model as a numerical reference for
the original Neural ODE control problem.

\begin{figure*}[ht]
    \centering
    \includegraphics[width=0.6\textwidth]
    {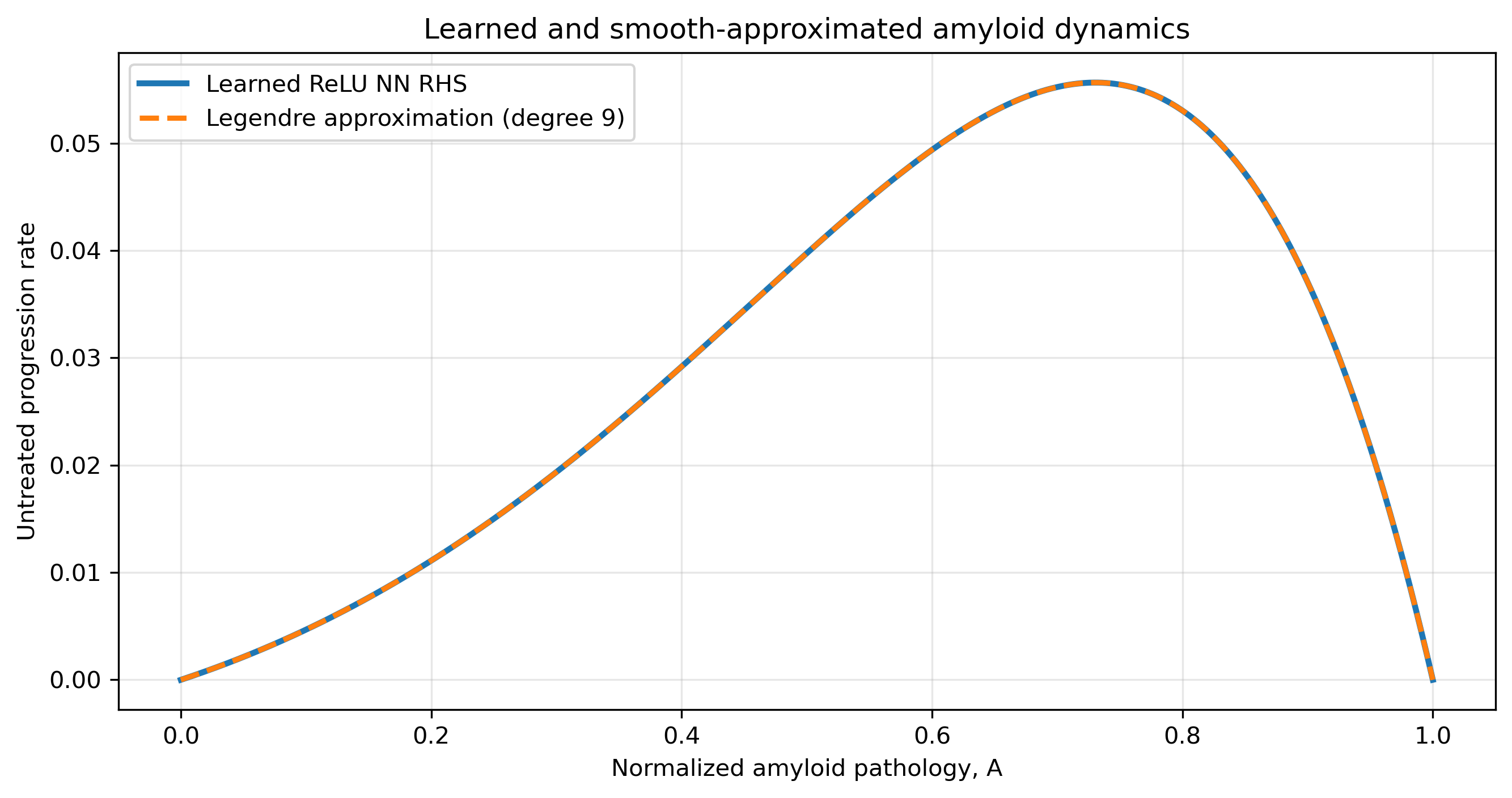}
    \caption{
    Learned amyloid right-hand side \(g_{\phi}(A)\) and its degree-\(9\)
    Legendre approximation \(g_L(A)\). The Legendre approximation is used
    only to construct the differentiable FBS reference.
    }
    \label{fig:amyloid-rhs-approximation}
\end{figure*}

\subsubsection{Optimal-control comparison}

We compare the two model--optimizer combinations
\begin{equation}
\label{eq:amyloid-comparison-methods}
\begin{aligned}
g_{\phi}
&\xrightarrow{\mathrm{HPDPO}}
\left(
u_{\mathrm{NN,HPDPO}},
A_{\mathrm{NN,HPDPO}}
\right),\\
g_L
&\xrightarrow{\mathrm{FBS}}
\left(
u_{L,\mathrm{FBS}},
A_{L,\mathrm{FBS}}
\right).
\end{aligned}
\end{equation}
The FBS solution serves as a reference for the smooth Legendre model; it is
not an exact solution of the original ReLU Neural ODE control problem.

Figure~\ref{fig:amyloid-control-comparison} compares the optimized treatment
policies and the resulting amyloid trajectories. The untreated Neural ODE
trajectory is included to illustrate the effect of treatment. In addition,
the FBS policy is evaluated under the original neural dynamics to provide a
common-model comparison between the two treatment policies.

\begin{figure*}[ht]
    \centering
    \includegraphics[width=\textwidth]
    {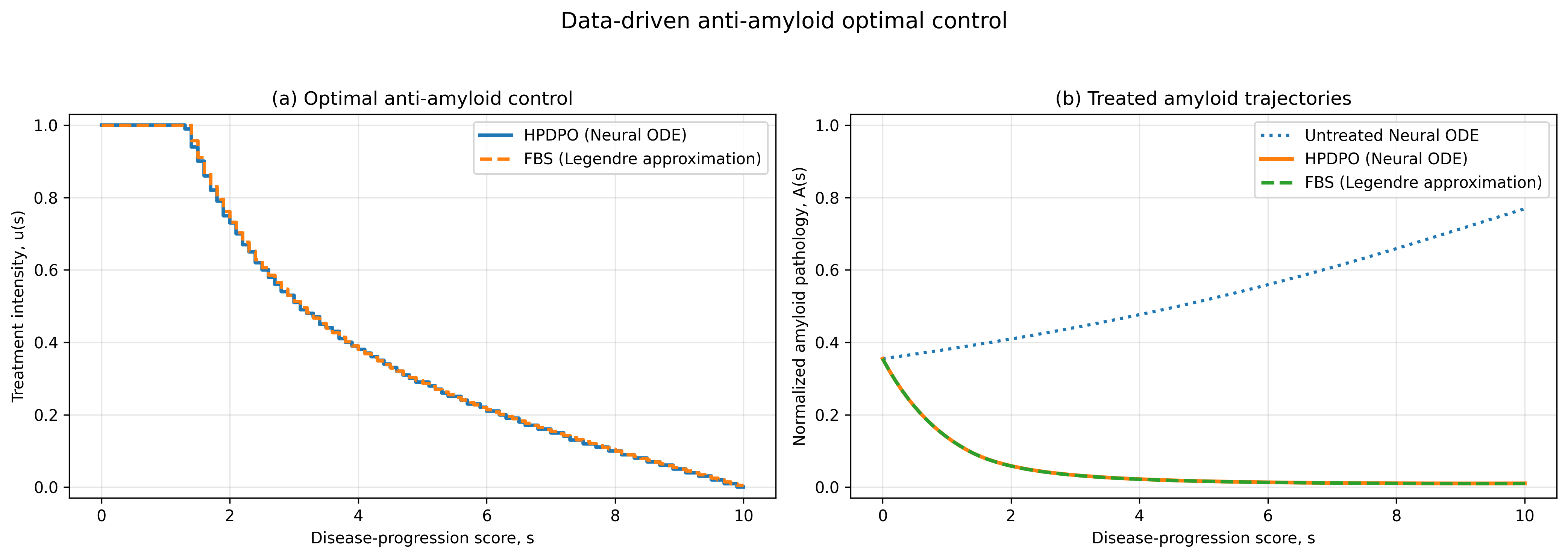}
    \caption{
    Data-driven anti-amyloid optimal control. Left: treatment policies
    obtained using HPDPO with the original Neural ODE and FBS with the smooth
    Legendre approximation. Right: untreated and treated amyloid
    trajectories. The FBS policy is additionally evaluated under the neural
    dynamics to provide a common-model comparison.
    }
    \label{fig:amyloid-control-comparison}
\end{figure*}

The numerical comparison is summarized in
Table~\ref{tab:amyloid-control-comparison}. The small right-hand-side
approximation error confirms that the degree-\(9\) Legendre expansion closely
reproduces the learned Neural ODE over the relevant state interval. The two
optimization procedures consequently produce similar treatment policies,
state trajectories, and objective values.

\begin{table*}[ht]
\centering
\caption{
Numerical comparison of the learned Neural ODE and its smooth Legendre
approximation, together with the corresponding HPDPO and FBS
optimal-control solutions for Example 4.
}
\label{tab:amyloid-control-comparison}
\begin{tabular}{lc}
\hline
Quantity & Value \\
\hline

RMSE of the Legendre approximation of the neural RHS
& \(3.80\times10^{-8}\) \\

RMSE between the optimized control policies
& \(4.55\times10^{-3}\) \\

RMSE between the resulting amyloid trajectories
& \(2.68\times10^{-4}\) \\

Objective value for the neural model with HPDPO control
& \(0.768578\) \\

Objective value for the Legendre model with FBS control
& \(0.768602\) \\

Absolute difference between the objective values
& \(2.38\times10^{-5}\) \\

Final amyloid value for the neural model with HPDPO control
& \(0.010284\) \\

Final amyloid value for the Legendre model with FBS control
& \(0.010016\) \\
\hline
\end{tabular}
\end{table*}

The close agreement between HPDPO and the FBS reference indicates that the
derivative-free optimization remains effective when applied directly to the
learned ReLU Neural ODE. In this example, HPDPO avoids replacing the
nondifferentiable learned dynamics during optimization, while producing a
control policy and amyloid trajectory that are close to those obtained from
the smooth reference model.
\newpage
\subsection{Multidimensional Example 1: Two-State Compartmental Problem with an Exact Solution}
\label{sec:multidim-automatica-two-control}

We first consider the finite-horizon compartmental control problem presented
by Blanchini et al.~\cite{Blanchini2023}. The state and control vectors are
\begin{equation}
x(t)
=
\begin{pmatrix}
x_1(t)\\
x_2(t)
\end{pmatrix}
\in\mathbb{R}^2,
\qquad
u(t)
=
\begin{pmatrix}
\alpha(t)\\
\beta(t)
\end{pmatrix}
\in\mathbb{R}^2.
\end{equation}
The problem is
\begin{equation}
\label{eq:multidim-automatica-problem}
\begin{aligned}
\min_{\alpha(\cdot),\beta(\cdot)}
\quad
J(\alpha,\beta)
&=
x_2(1),\\
\text{subject to}\quad
\dot{x}_1(t)
&=
-\alpha(t)x_1(t),\\
\dot{x}_2(t)
&=
\alpha(t)x_1(t)-\beta(t)x_2(t),\\
x(0)
&=
\begin{pmatrix}
1\\
0
\end{pmatrix},\\
1\leq\alpha(t)&\leq4,\qquad
2\leq\beta(t)\leq3,
\qquad 0\leq t\leq1.
\end{aligned}
\end{equation}
This example is especially useful for multidimensional validation because both
control components are optimized simultaneously while an analytical optimum is
available.

\subsubsection{Exact solution}

For the initial condition in Eq.~\eqref{eq:multidim-automatica-problem}, the exact
optimal control is bang--bang \cite{Blanchini2023}. The second component is
constant,
\begin{equation}
\label{eq:multidim-automatica-exact-beta}
\beta^*(t)=3,
\end{equation}
while the first component switches once:
\begin{equation}
\label{eq:multidim-automatica-exact-alpha}
\alpha^*(t)
=
\begin{cases}
4, & 0\leq t<t_s,\\
1, & t_s\leq t\leq1,
\end{cases}
\qquad
t_s
=
1-\frac{\ln 3}{2}
=
0.450693855666.
\end{equation}
For \(x(0)=(1,0)^{\mathsf T}\), the corresponding exact terminal objective is
\begin{equation}
\label{eq:multidim-automatica-exact-objective}
J^*
=
x_2^*(1)
=
0.103977422086.
\end{equation}

\subsubsection{Numerical comparison}

The interval \([0,1]\) is discretized using
\begin{equation}
\Delta t=0.01,
\qquad
M=100.
\end{equation}
For HPDPO, the two-dimensional action
\(u_k=(\alpha_k,\beta_k)^{\mathsf T}\) is optimized jointly at every time
point. We use a \(21\times21\) Cartesian-product grid over the full admissible
control rectangle, followed by two \(11\times11\) local joint refinements around
the current best action. The policy-smoothing weight is set to
\(\beta_u=0\), while
\(\beta_{\lambda,s}=1\) and
\(\beta_{\lambda,2}=0.05\).

To provide a direct reinforcement-learning comparison in the multidimensional
action setting, we also solve the same discretized problem using PPO and SAC.
All three numerical methods use the same initial condition, time grid, RK4
state transition, control bounds, and terminal objective. PPO and SAC use a
continuous two-dimensional action and are trained for \(250{,}000\)
environment transitions. After training, the learned policies are
evaluated deterministically. The control and state RMSE values are calculated
with respect to the exact solution sampled on the same time grid.

Table~\ref{tab:multidim-automatica-two-control} summarizes the results.
The computation times are wall-clock times measured on the same CPU
environment and include the optimization or training stage for each numerical
method.

\begin{table*}[ht]
\centering
\small
\caption{
Comparison of the exact solution, HPDPO, PPO, and SAC for the two-state,
two-control compartmental benchmark in Multidimensional Example 1.
}
\label{tab:multidim-automatica-two-control}
\begin{tabular}{lcccc}
\hline
Quantity
& Exact solution
& HPDPO
& PPO
& SAC \\
\hline

Objective \(J\)
& \(0.103977422086\)
& \(0.103977489605\)
& \(0.10516939\)
& \(0.12687182\) \\

Absolute objective difference
& \(0\)
& \(6.7519\times10^{-8}\)
& \(1.19197\times10^{-3}\)
& \(2.28944\times10^{-2}\) \\

Relative objective difference (\%)
& \(0\)
& \(6.4936\times10^{-5}\)
& \(1.14637124\)
& \(22.01862819\) \\

Control RMSE
& \(0\)
& \(0\)
& \(0.56527340\)
& \(0.90837831\) \\

State RMSE
& \(0\)
& \(1.5982\times10^{-4}\)
& \(1.642835\times10^{-2}\)
& \(3.393467\times10^{-2}\) \\

Switching time \(t_s\)
& \(0.450693855666\)
& \(0.45\)
& \(0.54\)
& \(0.72\) \\

Absolute switching-time error
& \(0\)
& \(6.9386\times10^{-4}\)
& \(8.930614\times10^{-2}\)
& \(2.6930614\times10^{-1}\) \\

Computation time (s)
& ---
& \(5.92\)
& \(93.70\)
& \(855.84\) \\

\hline
\end{tabular}
\end{table*}

HPDPO recovers the exact bang--bang control values,
\((4,3)\rightarrow(1,3)\), with the first control switching at
\(t=0.45\). The resulting objective,
\(J=0.103977489605\), differs from the analytical optimum by only
\(6.7519\times10^{-8}\), corresponding to a relative difference of
\(6.4936\times10^{-5}\%\). The switching-time error is
\(6.9386\times10^{-4}\), which is smaller than one control interval.

The PPO policy also approaches the exact terminal objective, with
\(J=0.10516939\) and a relative objective difference of approximately
\(1.15\%\), although its control and state trajectories deviate more from the
analytical solution. Its estimated switch occurs at \(t=0.54\). In this run,
SAC gives \(J=0.12687182\), corresponding to a relative difference of
approximately \(22.02\%\), with a later estimated switch at \(t=0.72\).

The measured wall-clock times were \(5.92\) s for HPDPO, \(93.70\) s for PPO,
and \(855.84\) s for SAC. Although computation time depends on the hardware,
software implementation, and training budget, the common CPU experiment shows
the additional training cost required by the neural-network RL baselines in
this benchmark. Together with the trajectory and objective comparisons, these
results provide a direct multidimensional validation of HPDPO against both the
known analytical optimum and standard continuous-control reinforcement-learning
methods.

\subsection{Multidimensional Example 2: Five-State Alzheimer's Disease Model with Two Treatment Controls}
\label{sec:multidim-alzheimer-cmo}

We next consider the controlled monomer--oligomer (CMO) model of Bulai et
al.~\cite{Bulai2025}. The model describes the interaction of amyloid monomers,
proto-oligomers, and oligomers and introduces two distinct simultaneous
treatment controls. The state and control vectors are
\begin{equation}
x(t)
=
\begin{pmatrix}
M(t)\\
U_2(t)\\
U_3(t)\\
U_4(t)\\
O(t)
\end{pmatrix}
\in\mathbb{R}^5,
\qquad
u(t)
=
\begin{pmatrix}
v(t)\\
w(t)
\end{pmatrix}
\in\mathbb{R}^2,
\end{equation}
where \(v(t)\) increases monomer degradation and \(w(t)\) increases oligomer
degradation. The symbol \(w\) is used here for the oligomer-degradation
control to avoid conflict with the vector notation \(u\); it corresponds to
the control denoted by \(u\) in \cite{Bulai2025}.

The controlled CMO dynamics used in the numerical experiment are
\begin{equation}
\label{eq:multidim-cmo-dynamics}
\begin{aligned}
\dot M={}&
sM\left(1-\frac{M}{K}\right)
-M(r_2U_2+r_3U_3+r_4U_4)
+(b+2\beta)U_3+(b+2\beta)U_4\\
&+2\beta U_2-a_{11}M^2-\delta M-vM,\\
\dot U_2={}&
-r_2MU_2-\beta U_2+(b+2\beta)U_3+2\beta U_4
-a_{23}U_2U_3-a_{24}U_2U_4\\
&-a_{22}U_2^2+a_{11}M^2,\\
\dot U_3={}&
-M(r_3U_3-r_2U_2)-(2\beta+b)U_3+(b+2\beta)U_4
-a_{23}U_2U_3\\
&-a_{34}U_3U_4-a_{33}U_3^2,\\
\dot U_4={}&
-M(r_4U_4-r_3U_3)-(b+3\beta)U_4
-a_{24}U_2U_4-a_{34}U_3U_4\\
&-a_{44}U_4^2+a_{22}U_2^2,\\
\dot O={}&
r_4MU_4+a_{23}U_2U_3+a_{24}U_2U_4+a_{34}U_3U_4
+a_{33}U_3^2\\
&+a_{44}U_4^2-m_cO^2-wO .
\end{aligned}
\end{equation}

Here, \(M(t)\) denotes the concentration of amyloid monomers,
\(U_2(t)\), \(U_3(t)\), and \(U_4(t)\) represent proto-oligomers
consisting of two, three, and four monomers, respectively, and \(O(t)\)
denotes the concentration of larger oligomers containing five or more
monomers. The model is nonlinear because of the quadratic aggregation terms,
cross-population interactions, and the state-dependent treatment terms
\(vM\) and \(wO\).

The initial condition is
\begin{equation}
\label{eq:multidim-cmo-initial-condition}
x(0)
=
\begin{pmatrix}
10^{-3}&0&0&0&0
\end{pmatrix}^{\mathsf T},
\qquad
T=1000\ {\rm s},
\end{equation}
and the treatment bounds are
\begin{equation}
\label{eq:multidim-cmo-control-bounds}
0\leq v(t)\leq0.03,
\qquad
0\leq w(t)\leq0.03.
\end{equation}

The parameter values are
\begin{equation}
\label{eq:multidim-cmo-parameters}
\begin{aligned}
&s=10^{-4},\quad K=1,\quad
r_2=r_3=r_4=100,\quad
b=10^{-3},\\
&\beta=10^{-4},\quad
\delta=5\times10^{-4},\quad
m_c=10^{-4},\\
&a_{11}=10,\quad
a_{22}=2.5,\quad
a_{23}=1.67,\quad
a_{24}=1.25,\\
&a_{33}=1.11,\quad
a_{34}=0.83,\quad
a_{44}=0.625.
\end{aligned}
\end{equation}

The CMO objective is
\begin{equation}
\label{eq:multidim-cmo-objective}
\begin{aligned}
J(v,w)
={}&
c_1M(T)+c_2O(T)\\
&+\int_0^T
\left[
c_3M+c_4O
+c_5v+c_6v^2
+c_7w+c_8w^2
+c_9vw
\right]\dd t,
\end{aligned}
\end{equation}
with
\begin{equation}
\label{eq:multidim-cmo-cost-weights}
c_1=c_2=c_3=c_4=1,\qquad
c_5=c_7=c_9=10^{-5},\qquad
c_6=c_8=2.
\end{equation}
The interaction term \(c_9vw\) explicitly couples the two treatment
components in the objective, providing a nonlinear biological benchmark with
joint multidimensional control.

\subsubsection{Numerical comparison}

Bulai et al.~\cite{Bulai2025} solve the CMO optimal-control problem using
Pontryagin's Maximum Principle together with a forward--backward sweep
method. Their procedure iteratively solves the nonlinear state equations
forward in time and the associated five-dimensional adjoint system backward
in time, followed by an update of the two controls from the corresponding
optimality conditions.

For HPDPO, the treatment policy is piecewise constant over intervals of
\(\Delta t=5\ {\rm s}\). Within each control interval, the nonlinear state
system is integrated using RK4 with internal substeps no larger than
\(0.5\ {\rm s}\). At each policy time, both controls are varied jointly using
a \(21\times21\) global Cartesian-product grid followed by two
\(11\times11\) local joint refinements.

The published CMO solution is
\begin{equation}
J_{\rm pub}=0.1148,
\qquad
\operatorname{dose}(v)=1.439,
\qquad
\operatorname{dose}(w)=2.306,
\end{equation}
where
\begin{equation}
\operatorname{dose}(v)=\int_0^T v(t)\,\dd t,
\qquad
\operatorname{dose}(w)=\int_0^T w(t)\,\dd t.
\end{equation}
As an implementation check, our no-treatment objective is
\(0.2528554150\), compared with the published value \(0.2529\), while the
maximum-treatment objective is \(3.6267422646\), compared with \(3.619\).

HPDPO obtains
\begin{equation}
J_{\rm HPDPO}=0.1149777157,
\end{equation}
corresponding to a relative difference of \(0.154805\%\) from the published
objective. The corresponding integrated treatment doses are
\begin{equation}
\operatorname{dose}(v)=1.445625,
\qquad
\operatorname{dose}(w)=2.309850.
\end{equation}

For comparison, PPO and SAC are applied to the same discretized control
problem. The observation consists of the five state variables together with
normalized time, and both treatment components are generated simultaneously
by the policy. Each RL method is trained on CPU with a fixed transition
budget. For PPO, the best deterministic checkpoint within the training budget
is retained.

\begin{table*}[ht]
\centering
\small
\caption{
Comparison of the published CMO solution with HPDPO, PPO, and SAC for the
five-state, two-control Alzheimer's disease benchmark, Multidimensional Example 2.}
\label{tab:multidim-alzheimer-cmo-comparison}
\begin{tabular}{lcccc}
\hline
Quantity
& Published CMO
& HPDPO
& PPO
& SAC \\
\hline

Objective \(J\)
& \(0.114800\)
& \(0.114978\)
& \(0.146052\)
& \(0.118328\) \\

Relative objective difference (\%)
& \(0\)
& \(0.1548\)
& \(27.2227\)
& \(3.0731\) \\

Dose of monomer treatment
& \(1.439\)
& \(1.4456\)
& \(2.9783\)
& \(2.1902\) \\

Dose of oligomer treatment
& \(2.306\)
& \(2.3099\)
& \(1.1837\)
& \(2.6865\) \\

Computation time (s)
& ---
& \(1569.1\)
& \(335.1\)
& \(1003.7\) \\

\hline
\end{tabular}
\end{table*}

The comparison reveals different numerical behavior across the three
methods. HPDPO closely reproduces both the published objective and the total
allocation of the two treatments, despite not using the problem-specific
adjoint equations or the analytical control-update formulas required by the
forward--backward method. SAC also approaches the published objective, with a
relative difference of approximately \(3.1\%\), although its integrated
treatment doses differ more noticeably from the reference.

PPO is more sensitive to the learning configuration in this example. In our
experiments, increasing the GAE parameter to one and including a small
entropy term substantially improved the learned policy; nevertheless, its
best objective remains \(27.2\%\) above the published value and the relative
allocation of the two treatments differs from the forward--backward
solution. This behavior is consistent with the long-horizon nature of the
problem, in which each treatment decision influences the subsequent
nonlinear evolution of several coupled state variables.

The wall-clock results also illustrate the different computational
tradeoffs. PPO is the fastest of the tested implementations, whereas HPDPO
requires more computation because each policy update evaluates a large number
of joint treatment candidates through deterministic model rollouts. In
return, HPDPO provides the closest agreement with the published
forward--backward solution in both objective value and integrated treatment
dose. Thus, this example demonstrates the ability of HPDPO to recover a
high-accuracy solution for a nonlinear, coupled, multidimensional control
problem without requiring explicit derivatives of the state dynamics.

\subsection{Multidimensional Example 3: Five-control flood-management problem}
\label{sec:multidim-flood}

We next consider the six-state, five-control compartmental flood-management
problem studied by Blanchini et al.~\cite{Blanchini2023}. This example is used
to test HPDPO in a higher-dimensional control space, where a dense Cartesian
grid becomes computationally impractical.

The state and control vectors are
\[
x(t)=(x_1(t),\ldots,x_6(t))^\top\in\mathbb{R}^6,
\qquad
v(t)=(v_1(t),\ldots,v_5(t))^\top\in\mathbb{R}^5.
\]
The dynamics are given by
\begin{equation}
\label{eq:multidim-flood-dynamics}
\dot{x}(t)
=
\widetilde{A}x(t)
+
\widetilde{F}\operatorname{diag}(v(t))
\widetilde{G}x(t),
\end{equation}
where
\[
\widetilde{A}
=
\begin{pmatrix}
-0.1 & 0 & 0 & 0 & 0 & 0\\
0 & -0.1 & 0 & 0 & 0 & 0\\
0 & 0 & -0.2 & 0 & 0 & 0\\
0 & 0 & 0 & -0.2 & 0 & 0\\
0 & 0 & 0.1 & 0 & -0.1 & 0\\
0 & 0 & 0 & 0.1 & 0 & -0.2
\end{pmatrix},
\]
\[
\widetilde{F}
=
\begin{pmatrix}
-1 & -1 & -1 & 0 & 0\\
1 & 0 & 0 & -1 & 0\\
0 & 1 & 0 & 0 & -1\\
0 & 0 & 1 & 1 & 1\\
0 & 0 & 0 & 0 & 0\\
0 & 0 & 0 & 0 & 0
\end{pmatrix},
\qquad
\widetilde{G}
=
\begin{pmatrix}
1 & 0 & 0 & 0 & 0 & 0\\
1 & 0 & 0 & 0 & 0 & 0\\
1 & 0 & 0 & 0 & 0 & 0\\
0 & 1 & 0 & 0 & 0 & 0\\
0 & 0 & 1 & 0 & 0 & 0
\end{pmatrix}.
\]
Here, $\widetilde{A}$ describes the uncontrolled compartmental
dynamics, while $\widetilde{F}$ and $\widetilde{G}$ determine the
controlled flows between compartments. The initial condition and
time horizon are
\[
x(0)=(1,0,0,0,0,0)^\top,
\qquad T=12.
\]

The admissible control set is the five-dimensional box
\[
\mathcal{V}
=
\prod_{j=1}^{5}[v_j^-,v_j^+],
\]
with
\[
v^-=(0.15,0.16,0.17,0.18,0.19)^\top,
\qquad
v^+=(1.15,1.16,1.17,1.18,1.19)^\top.
\]
The objective is
\begin{equation}
\label{eq:multidim-flood-objective}
\min_{v(\cdot)\in\mathcal{V}} J(v)
=
\int_0^{12}
\left(x_1(t)+1.2x_5(t)+x_6(t)\right)\,dt.
\end{equation}
Unlike the previous two-control examples, a tensor-product action grid is
not practical here. For example, using 21 candidate values for each control
component would require
\[
21^5=4{,}084{,}101
\]
joint control vectors at every time point. We therefore use the dimension-adaptive derivative-free action-search
strategy described in the method section. At each time step, the candidate
set consists of a joint Sobol/QMC sample over the full
five-dimensional admissible control box, augmented by all
\[
2^5=32
\]
corners of the box. The best candidate is then refined using two local
five-dimensional Sobol/QMC searches in shrinking hyperrectangles.

In the reported experiment, the global search used 512 Sobol candidates and
each local refinement used 256 Sobol candidates. All five control components
were varied jointly. The temporal
step size was
\[
\Delta t=0.1.
\]

For comparison, we evaluate the HJB-derived reference control from
Blanchini et al.~\cite{Blanchini2023} on the same temporal grid. This gives
the reference value
\begin{equation}
J_{\mathrm{HJB}}
=
1.590377645860.
\end{equation}
\begin{table*}[ht]
\centering
\small
\caption{
Comparison of the same-grid HJB reference, HPDPO, PPO, and SAC for Multidimensional Example 3. The HJB reference is
evaluated on the same \(\Delta t=0.1\) temporal grid.}
\label{tab:flood_hpdpo}
\begin{tabular}{lcccc}
\hline
Quantity
& HJB reference
& HPDPO
& PPO
& SAC \\
\hline

Objective \(J\)
& \(1.5903776459\)
& \(1.5903775969\)
& \(1.5904215236\)
& \(1.6056206441\) \\

Relative objective difference (\%)
& \(0\)
& \(-3.08\times10^{-6}\)
& \(2.759\times10^{-3}\)
& \(9.58451\times10^{-1}\) \\

State RMSE
& \(0\)
& \(4.14\times10^{-6}\)
& \(2.59\times10^{-5}\)
& \(4.60\times10^{-3}\) \\

Control RMSE
& \(0\)
& \(9.34\times10^{-2}\)
& \(3.53\times10^{-1}\)
& \(5.20\times10^{-1}\) \\

Computation time (s)
& ---
& \(53.02\)
& \(186.26\)
& \(3978.21\) \\

\hline
\end{tabular}
\end{table*}
To provide a direct reinforcement-learning comparison in the
five-dimensional action space, we also solve the same discretized problem
using PPO and SAC. All three numerical methods use the same initial
condition, time horizon, RK4 state transition, control bounds, and objective.
For PPO and SAC, the observation consists of the six-dimensional state
together with normalized time, and the policy outputs all five control
components simultaneously. Both reinforcement-learning methods are trained
for \(250{,}000\) environment transitions on CPU and are evaluated
deterministically after training.

HPDPO attained
\begin{equation}
J_{\mathrm{HPDPO}}
=
1.5903775969,
\end{equation}
corresponding to a relative difference of approximately
\begin{equation}
-3.08\times10^{-6}\%
\end{equation}
from the same-grid HJB reference. The state RMSE is
\(4.14\times10^{-6}\), while the control RMSE is
\(9.34\times10^{-2}\).

PPO also gives an objective close to the HJB reference,
\begin{equation}
J_{\mathrm{PPO}}
=
1.5904215236,
\end{equation}
with a relative objective difference of \(0.002759\%\). The corresponding
state and control RMSE values are \(2.59\times10^{-5}\) and
\(3.53\times10^{-1}\), respectively.

For SAC, the obtained objective is
\begin{equation}
J_{\mathrm{SAC}}
=
1.6056206441,
\end{equation}
corresponding to a relative difference of \(0.958451\%\). Its state RMSE is
\(4.60\times10^{-3}\), and its control RMSE is \(5.20\times10^{-1}\).

Table~\ref{tab:flood_hpdpo} summarizes the comparison. Computation times are
wall-clock times measured in the same CPU environment for HPDPO, PPO, and
SAC. A computation time for the published HJB reference is not included.
HPDPO reproduces the HJB objective and state trajectory to numerical
precision on the common temporal grid. PPO also closely matches the reference
objective and state trajectory, although its control RMSE is larger. This
indicates that different five-dimensional control profiles can produce very
similar state trajectories and accumulated objective values in this problem.
SAC shows a larger deviation in both the control and state trajectories but
still gives an objective within approximately \(0.96\%\) of the HJB
reference in this run.

The measured wall-clock times were \(53.02\) s for HPDPO, \(186.26\) s for
PPO, and \(3978.21\) s for SAC. These timings depend on the hardware,
implementation, and training budget, but the common CPU experiment provides
a direct computational comparison under the settings used here.

The recovered HPDPO control structure is also consistent with the HJB
reference. The fourth control remains at its lower bound throughout the
horizon, while the fifth control remains essentially at its upper bound.
The second and third controls recover the expected lower-to-upper switching
behavior. On the \(\Delta t=0.1\) grid, HPDPO changes \(v_2\) near
\(t=7.0\)--\(7.1\), compared with the HJB switch at \(t=7.1\), and changes
\(v_3\) near \(t=5.0\), compared with the HJB switch at \(t=5.1\).
These differences are within one temporal grid interval.
\section{Discussion and Conclusions}
\label{sec:discussion-conclusion}

This work introduced HPDPO, a finite-horizon policy-optimization method that combines deterministic model rollout with a regularized Hamiltonian primal--dual objective. The method evaluates candidate controls through forward simulation and updates
the multiplier sequence through bounded block-coordinate maximization, avoiding
both differentiation of the transition dynamics and approximation of a global
value function.

Across two nonsmooth benchmarks, an analytical and learned nonlinear cascade, and an ADNI-trained Neural ODE, HPDPO produced control and state trajectories that closely agreed with exact solutions or derivative-based references. In the two benchmark comparisons, it also achieved smaller errors and substantially shorter computation times than the tested PPO and SAC implementations. The three additional multidimensional-control examples show that the same formulation can optimize simultaneous control components. HPDPO reproduces the exact switching structure of the two-state compartmental problem, closely matches the published objective and treatment doses of the five-state Alzheimer's disease model, and matches the same-grid HJB reference for the six-state, five-control flood problem using the QMC action search. These results support HPDPO as a practical approach for deterministic finite-horizon problems in which the transition model is reliable for simulation but unsuitable for derivative-based optimization.
The present implementation produces a trajectory-specific control sequence for
a prescribed initial condition; changes in the observed state would require
re-optimization or a receding-horizon implementation. Its computational cost
also increases with the horizon, candidate-set size, and control dimension.
For a fixed policy, the dual block-coordinate sweep converges to the unique
bounded dual maximizer under the conditions of
Proposition~\ref{prop:dual-coordinate-convergence}. Convergence of the full
alternating primal--dual iteration, however, remains to be established.
Future work will therefore focus on receding-horizon implementation, scalable
action search, and theoretical analysis of the full discrete saddle iteration.

\subsection*{Author Contributions}
\textit{K. Rabiei developed the methodology, implemented the algorithms, conducted the numerical experiments, analyzed the results, and drafted the manuscript. Z. Rui contributed to methodological development, scientific discussions, and manuscript review. X. Lu contributed to reviewing, editing, and checking the manuscript. W. Hao contributed to the conceptualization of the study, supervised the research, provided methodological guidance, and reviewed and edited the manuscript. All authors reviewed and approved the final manuscript.}

\subsection*{Acknowledgments}
\textit{K. R. and W. H. were supported by the National Institute of General Medical Sciences (NIGMS) through Grant No. R35GM146894. W. H. was also supported by the National Science Foundation (NSF) through Grant No. DMS-2533995 and by the Huck Chair in AI Mathematical Modeling from The Pennsylvania State University’s Huck Institutes of the Life Sciences.}

\subsection*{Conflicts of Interest}
\textit{The authors declare that they have no conflicts of interest.
}
\subsection*{Use of Artificial Intelligence Tools}

\textit{ChatGPT (OpenAI) was used for language editing, improving clarity, and organizing portions of the manuscript. All scientific content, analyses, interpretations, and final decisions were reviewed and approved by the authors.}

\subsection*{Data Availability Statement}
\textit{The source code supporting the findings of this study is available at the following GitHub repository: https://github.com/LRabiei/HPDPO-derivative-free-control}

\clearpage
\appendix
\section*{Supplementary Information}
\setcounter{section}{0}
\setcounter{equation}{0}
\setcounter{figure}{0}
\setcounter{table}{0}
\renewcommand{\thesection}{S\arabic{section}}
\renewcommand{\thesubsection}{S\arabic{section}.\arabic{subsection}}
\renewcommand{\theequation}{S\arabic{equation}}
\renewcommand{\theHequation}{supp.\arabic{equation}}
\renewcommand{\thefigure}{S\arabic{figure}}
\renewcommand{\thetable}{S\arabic{table}}
\section{Theoretical Proofs}
The statements and proofs below provide the detailed derivations supporting the theoretical results summarized in the main Research Article.

\subsection{Dual-block representation}
$\Lambda=(\lambda_0,\ldots,\lambda_M)$ as the stacked column vector
\begin{equation}
\Lambda
=
\begin{pmatrix}
\lambda_0\\
\lambda_1\\
\vdots\\
\lambda_M
\end{pmatrix}
\in\mathbb R^{(M+1)n}.
\end{equation}
To represent the consecutive differences appearing in the multiplier
smoothness penalty, define the first-difference matrix
\begin{equation}
D_M
=
\begin{pmatrix}
-1 & 1 & 0 & \cdots & 0\\
0 & -1 & 1 & \cdots & 0\\
\vdots & & \ddots & \ddots & \vdots\\
0 & \cdots & 0 & -1 & 1
\end{pmatrix}
\in\mathbb R^{M\times(M+1)},
\end{equation}
and its block extension
\begin{equation}
\mathcal D
=
D_M\otimes I_n
\in
\mathbb R^{Mn\times(M+1)n},
\end{equation}
where $I_n$ denotes the $n\times n$ identity matrix. The matrix $D_M$
acts on the time index, while the Kronecker product with $I_n$ applies the
same first-difference operation to the $n$ components of each multiplier
block. Hence,
\begin{equation}
\mathcal D\Lambda
=
\begin{pmatrix}
\lambda_1-\lambda_0\\
\lambda_2-\lambda_1\\
\vdots\\
\lambda_M-\lambda_{M-1}
\end{pmatrix},
\end{equation}
and therefore
\begin{equation}
\|\mathcal D\Lambda\|^2
=
\sum_{k=0}^{M-1}
\|\lambda_{k+1}-\lambda_k\|^2.
\end{equation}
Similarly,
\begin{equation}
\|\Lambda\|^2
=
\sum_{k=0}^{M}\|\lambda_k\|^2.
\end{equation}
Consequently, for fixed $U$, the terms of $L_{\mathrm{reg}}$ that depend on
$\Lambda$ can be written compactly as
\begin{equation}
\label{eq:dual-compact-form}
\sum_{k=0}^{M}\lambda_k^Tb_k(U)
-\beta_{\lambda,s}\|\mathcal D\Lambda\|^2
-\beta_{\lambda,2}\|\Lambda\|^2.
\end{equation}
This representation makes the curvature of the dual subproblem explicit. For a twice-differentiable function, a uniform negative-definite Hessian
bound provides a standard characterization of strong concavity
\cite{BoydVandenberghe2004}.

\subsection{Strong concavity of the dual block}
\begin{lemma}[Strong concavity of the dual block]
\label{lem:dual-strong-concavity}
For every fixed admissible policy $U$, if $\beta_{\lambda,2}>0$, then
$L_{\mathrm{reg}}(U,\cdot)$ is
$2\beta_{\lambda,2}$-strongly concave with respect to the multiplier
sequence $\Lambda$.
\end{lemma}

\begin{proof}
For fixed $U$, the trajectory $X=\mathcal R(U)$ and the coefficients
$b_k(U)$ are fixed. By Eq.~\eqref{eq:dual-compact-form}, the
$\Lambda$-dependent part of $L_{\mathrm{reg}}$ consists of the linear term
\[
\sum_{k=0}^{M}\lambda_k^Tb_k(U),
\]
together with the quadratic terms
\[
-\beta_{\lambda,s}\|\mathcal D\Lambda\|^2
\qquad\text{and}\qquad
-\beta_{\lambda,2}\|\Lambda\|^2.
\]
The linear term has zero Hessian. Since
\[
\|\mathcal D\Lambda\|^2
=
\Lambda^T\mathcal D^T\mathcal D\Lambda,
\]
the two quadratic terms have Hessians
$-2\beta_{\lambda,s}\mathcal D^T\mathcal D$ and
$-2\beta_{\lambda,2}I_{(M+1)n}$, respectively. Therefore,
\begin{equation}
\nabla_{\Lambda\Lambda}^2 L_{\mathrm{reg}}
=
-2\beta_{\lambda,s}\mathcal D^T\mathcal D
-2\beta_{\lambda,2}I_{(M+1)n}.
\end{equation}

For any $z\in\mathbb R^{(M+1)n}$,
\[
z^T\mathcal D^T\mathcal D z
=
\|\mathcal D z\|^2
\geq 0,
\]
so $\mathcal D^T\mathcal D$ is positive semidefinite. Since
$\beta_{\lambda,s}\geq0$, it follows that
\begin{equation}
\nabla_{\Lambda\Lambda}^2 L_{\mathrm{reg}}
\preceq
-2\beta_{\lambda,2}I_{(M+1)n}.
\end{equation}
Hence, when $\beta_{\lambda,2}>0$,
$L_{\mathrm{reg}}(U,\cdot)$ is
$2\beta_{\lambda,2}$-strongly concave \cite{BoydVandenberghe2004}.
\end{proof}

\subsection{Existence and uniqueness of the dual maximizer}
\begin{proposition}[Existence and uniqueness of the dual maximizer]
\label{prop:dual-existence-uniqueness}
Let $U$ be any fixed admissible policy and suppose that
$\beta_{\lambda,2}>0$. Then the dual subproblem
\begin{equation}
\max_{\Lambda\in\mathcal B_\lambda}
L_{\mathrm{reg}}(U,\Lambda)
\end{equation}
admits a unique global maximizer, denoted by
\begin{equation}
\Lambda^\star(U)
=
\argmax_{\Lambda\in\mathcal B_\lambda}
L_{\mathrm{reg}}(U,\Lambda).
\end{equation}
\end{proposition}

\begin{proof}
For fixed $U$, $L_{\mathrm{reg}}(U,\cdot)$ is continuous on the nonempty
compact set $\mathcal B_\lambda$. By the Extreme Value Theorem, it therefore
attains a maximum on $\mathcal B_\lambda$ \cite{Rudin1976}.
By Lemma~\ref{lem:dual-strong-concavity}, it is strongly concave when
$\beta_{\lambda,2}>0$. Since $\mathcal B_\lambda$ is convex, the maximizer
is unique.
\end{proof}

\subsection{Consistency of the finite action search}
\begin{corollary}[Consistency of the finite action search]
\label{cor:finite-action-search}
Fix $U$, $\Lambda$, and a policy coordinate $k$. Suppose that there exists
$L_Q>0$ such that
\begin{equation}
\left|
\mathcal Q_{k,h}^{\mathrm{HP}}(v\mid U,\Lambda)
-
\mathcal Q_{k,h}^{\mathrm{HP}}(w\mid U,\Lambda)
\right|
\leq
L_Q \|v-w\|_2
\end{equation}
for all $v,w\in\mathcal U$. Then
\begin{equation}
0
\leq
\min_{v\in\mathcal V_\delta}
\mathcal Q_{k,h}^{\mathrm{HP}}(v\mid U,\Lambda)
-
\min_{v\in\mathcal U}
\mathcal Q_{k,h}^{\mathrm{HP}}(v\mid U,\Lambda)
\leq
L_Q\delta .
\end{equation}
Consequently, the minimum score over $\mathcal V_\delta$ converges to
the continuous coordinate minimum as $\delta\to0$.
\end{corollary}

\begin{proof}
Since $\mathcal U$ is compact and the Lipschitz assumption implies
continuity, the Extreme Value Theorem guarantees that
$\mathcal Q_{k,h}^{\mathrm{HP}}(\cdot\mid U,\Lambda)$ attains its minimum
on $\mathcal U$ \cite{Rudin1976}. Let $v^\star\in\mathcal U$ denote such
a minimizer. By the definition of the covering radius $\delta$, there
exists $w_\delta\in\mathcal V_\delta$ such that
\[
\|w_\delta-v^\star\|_2\leq\delta.
\]
Since $\mathcal V_\delta\subseteq\mathcal U$,
\[
\begin{aligned}
0
&\leq
\min_{v\in\mathcal V_\delta}
\mathcal Q_{k,h}^{\mathrm{HP}}(v\mid U,\Lambda)
-
\mathcal Q_{k,h}^{\mathrm{HP}}(v^\star\mid U,\Lambda)\\
&\leq
\mathcal Q_{k,h}^{\mathrm{HP}}(w_\delta\mid U,\Lambda)
-
\mathcal Q_{k,h}^{\mathrm{HP}}(v^\star\mid U,\Lambda)\\
&\leq
L_Q\|w_\delta-v^\star\|_2
\leq
L_Q\delta .
\end{aligned}
\]
Because
\[
\mathcal Q_{k,h}^{\mathrm{HP}}(v^\star\mid U,\Lambda)
=
\min_{v\in\mathcal U}
\mathcal Q_{k,h}^{\mathrm{HP}}(v\mid U,\Lambda),
\]
the stated bound follows. Letting $\delta\to0$ gives convergence of the
minimum score over $\mathcal V_\delta$ to the continuous coordinate
minimum.
\end{proof}

\subsection{Multiplier block update}
\paragraph{Backward Hamiltonian dual sweep}

After the policy sweep, the updated policy $U^{(r+1)}$ and its rollout
$X^{(r+1)}=\mathcal R(U^{(r+1)})$ are fixed. The multiplier blocks are
then maximized in the default backward order
\begin{equation}
\label{eq:backward-dual-order}
\lambda_M
\longrightarrow
\lambda_{M-1}
\longrightarrow
\cdots
\longrightarrow
\lambda_0.
\end{equation}
Assuming $\beta_{\lambda,2}>0$, fix $U$, and hence
$X=\mathcal R(U)$ and $b_k(U)$, and consider an interior multiplier block
$1\leq k\leq M-1$. The terms of
$L_{\mathrm{reg}}$ involving $\lambda_k$ are
\begin{equation}
\begin{aligned}
b_k(U)^T\lambda_k
&-\beta_{\lambda,s}\|\lambda_k-\lambda_{k-1}\|^2\\
&-\beta_{\lambda,s}\|\lambda_{k+1}-\lambda_k\|^2
-\beta_{\lambda,2}\|\lambda_k\|^2 .
\end{aligned}
\end{equation}
Differentiating with respect to $\lambda_k$ gives
\begin{equation}
b_k(U)
+2\beta_{\lambda,s}(\lambda_{k-1}+\lambda_{k+1})
-(4\beta_{\lambda,s}+2\beta_{\lambda,2})\lambda_k.
\end{equation}
Setting this expression equal to zero gives the unconstrained maximizer of
the interior block subproblem. The endpoint expressions follow analogously
by retaining only the single neighboring smoothness term. Consequently,
\begin{equation}
\label{eq:lambda-closed-form}
\widehat{\lambda}_k
=
\begin{cases}
\displaystyle
\frac{
b_0(U)+2\beta_{\lambda,s}\lambda_1
}{
2\beta_{\lambda,s}+2\beta_{\lambda,2}
},
& k=0,
\\[4mm]
\displaystyle
\frac{
b_k(U)+2\beta_{\lambda,s}
(\lambda_{k-1}+\lambda_{k+1})
}{
4\beta_{\lambda,s}+2\beta_{\lambda,2}
},
& 1\leq k\leq M-1,
\\[4mm]
\displaystyle
\frac{
b_M(U)+2\beta_{\lambda,s}\lambda_{M-1}
}{
2\beta_{\lambda,s}+2\beta_{\lambda,2}
},
& k=M.
\end{cases}
\end{equation}

Since the block objective is a concave quadratic centered at
$\widehat{\lambda}_k$, maximizing it over
$[-\lambda_{\max},\lambda_{\max}]^n$ is equivalent to projecting
$\widehat{\lambda}_k$ onto that box. Hence, the exact constrained
block maximizer is
\begin{equation}
\label{eq:lambda-coordinate-update}
\lambda_k^{+}
=
\Pi_{[-\lambda_{\max},\lambda_{\max}]^n}
\left(\widehat{\lambda}_k\right).
\end{equation}
Thus, each multiplier block is updated by its exact constrained maximizer.

\subsection{Convergence of the dual block-coordinate sweep}
\begin{proposition}[Convergence of the dual block-coordinate sweep]
\label{prop:dual-coordinate-convergence}
Fix an admissible policy $U$ and suppose that
$\beta_{\lambda,2}>0$. If the multiplier blocks are repeatedly updated
in a fixed cyclic order using the exact constrained block maximizer in
Eq.~\eqref{eq:lambda-coordinate-update}, then the resulting multiplier
sequence converges to the unique solution $\Lambda^\star(U)$ of the
fixed-$U$ dual maximization problem.
\end{proposition}

\begin{proof}
For fixed $U$, define
\[
F_U(\Lambda)=-L_{\mathrm{reg}}(U,\Lambda).
\]
By Lemma~\ref{lem:dual-strong-concavity}, $F_U$ is continuously
differentiable and strongly convex on $\mathcal B_\lambda$.
Moreover,
\[
\mathcal B_\lambda
=
\prod_{k=0}^{M}
[-\lambda_{\max},\lambda_{\max}]^n
\]
is a compact Cartesian product of convex sets.

The update in Eq.~\eqref{eq:lambda-coordinate-update} is the exact
minimization of $F_U$ with respect to one multiplier block while all
remaining blocks are held fixed. Hence the multiplier sweep is a cyclic
block nonlinear Gauss--Seidel method applied to
\[
\min_{\Lambda\in\mathcal B_\lambda}F_U(\Lambda).
\]
Since $F_U$ is strongly convex, it is pseudoconvex, and its level sets
relative to $\mathcal B_\lambda$ are compact. Therefore, the convergence
result of Grippo and Sciandrone \cite[Proposition~6]{GrippoSciandrone2000}
implies that every limit point of the multiplier sequence is a global
minimizer of $F_U$.

Strong convexity implies that this minimizer is unique. By
Proposition~\ref{prop:dual-existence-uniqueness}, it is
$\Lambda^\star(U)$. Hence the multiplier sequence converges to
$\Lambda^\star(U)$.
\end{proof}

\section{Detailed Analytical Solution for the Constructed Nonsmooth Benchmark, Example 1}
We first consider a scalar optimal-control problem whose optimal trajectory
crosses the nondifferentiable point \(x=0\). Since a closed-form solution is
available, this example allows direct comparison of the computed control,
state, crossing time, and objective value with their exact counterparts.

The problem is
\begin{equation}
\label{eq:constructed-problem}
\begin{aligned}
\min_{u(\cdot)}\quad
J(u)
&=
\frac{\gamma}{2}\bigl(x(2)-x_d\bigr)^2
+\frac{1}{2}\int_{0}^{2}u(t)^2\,dt,\\
\text{subject to}\quad
\dot{x}(t)
&=
|x(t)|+u(t),\\
x(0)&=-a,\\
0&\leq u(t)\leq u_{\max},
\end{aligned}
\end{equation}
where \(a>0\), \(\gamma>0\), and
\begin{equation}
\label{eq:constructed-xd}
x_d
=
a+\frac{2a}{\gamma(e^2-1)}.
\end{equation}
The target \(x_d\) is selected so that the optimal trajectory crosses
\(x=0\), where the dynamics are continuous but nondifferentiable. For the numerical experiment, we use
\begin{equation}
\label{eq:constructed-parameters}
a=0.2,
\qquad
\gamma=1,
\qquad
u_{\max}=0.2.
\end{equation}

Let \(\tau\in(0,2)\) denote the zero-crossing time and let
\(y=x(2)\) denote the terminal state. For a fixed pair \((\tau,y)\), the
minimum control effort can be determined separately before and after the
crossing.

\paragraph*{Minimum effort before the crossing.}
\hspace{3mm} For \(0\leq t\leq\tau\), the state is negative, and the dynamics reduce to
\begin{equation}
\dot{x}=-x+u,
\qquad
x(0)=-a,
\qquad
x(\tau)=0.
\end{equation}
Multiplication by the integrating factor \(e^t\) gives
\begin{equation}
\frac{d}{dt}\bigl(e^t x(t)\bigr)=e^t u(t).
\end{equation}
Integrating over \([0,\tau]\) and using the endpoint conditions yields
\begin{equation}
\label{eq:pre-crossing-constraint}
\int_0^\tau e^t u(t)\,dt=a.
\end{equation}
The Cauchy--Schwarz inequality then gives
\begin{equation}
a^2
=
\left(\int_0^\tau e^t u(t)\,dt\right)^2
\leq
\left(\int_0^\tau e^{2t}\,dt\right)
\left(\int_0^\tau u(t)^2\,dt\right).
\end{equation}
Since
\begin{equation}
\int_0^\tau e^{2t}\,dt
=
\frac{e^{2\tau}-1}{2},
\end{equation}
we obtain
\begin{equation}
\frac{1}{2}\int_0^\tau u(t)^2\,dt
\geq
\frac{a^2}{e^{2\tau}-1}.
\end{equation}

Equality holds for \(u(t)\propto e^t\), yielding the minimum pre-crossing effort
\begin{equation}
E_-(\tau)
:=
\min_{u}
\left\{
\frac{1}{2}\int_0^\tau u(t)^2\,dt
:
x(0)=-a,\;
x(\tau)=0
\right\}
=
\frac{a^2}{e^{2\tau}-1},
\end{equation}
and
\begin{equation}
\label{eq:u-minus-general}
u_-(t)
=
\frac{2ae^t}{e^{2\tau}-1},
\qquad
0\leq t\leq\tau.
\end{equation}

\paragraph*{Minimum effort after the crossing.}
\hspace{3mm}For \(\tau\leq t\leq2\), the state is positive, so
\begin{equation}
\dot{x}=x+u,
\qquad
x(\tau)=0,
\qquad
x(2)=y.
\end{equation}
Multiplication by the integrating factor \(e^{-t}\) and integration over
\([\tau,2]\) give
\begin{equation}
\label{eq:post-crossing-constraint}
y
=
\int_\tau^2 e^{2-t}u(t)\,dt.
\end{equation}
Applying the same Cauchy--Schwarz argument used above yields
\begin{equation}
\label{eq:E-plus}
E_+(\tau,y)
=
\frac{y^2}{e^{2(2-\tau)}-1},
\end{equation}
with minimizing control
\begin{equation}
\label{eq:u-plus-general}
u_+(t)
=
\frac{2y\,e^{2-t}}{e^{2(2-\tau)}-1},
\qquad
\tau\leq t\leq2.
\end{equation}

Combining the pre- and post-crossing contributions reduces the original
optimal-control problem to
\begin{equation}
\label{eq:reduced-J}
\overline{J}(\tau,y)
=
\frac{a^2}{e^{2\tau}-1}
+
\frac{y^2}{e^{2(2-\tau)}-1}
+
\frac{\gamma}{2}(y-x_d)^2.
\end{equation}
For a fixed crossing time \(\tau\), minimization with respect to \(y\) gives
\begin{equation}
\label{eq:y-of-tau}
y(\tau)
=
\frac{
\gamma x_d\bigl(e^{2(2-\tau)}-1\bigr)
}{
\gamma\bigl(e^{2(2-\tau)}-1\bigr)+2
}.
\end{equation}

It remains to determine the optimal crossing time. Define
\begin{equation}
r=e^{2\tau},
\qquad
s=e^2.
\end{equation}
Substituting \eqref{eq:y-of-tau} into \eqref{eq:reduced-J} and using
\eqref{eq:constructed-xd} gives
\begin{equation}
\label{eq:global-difference}
\overline{J}\bigl(\tau,y(\tau)\bigr)
-
\overline{J}(1,a)
=
\frac{
a^2(r-s)^2\bigl(\gamma s^2-\gamma+2\bigr)
}{
(r-1)(s-1)^2
\left[\gamma(s^2-r)+2r\right]
}.
\end{equation}
For \(0<\tau<2\), we have \(1<r<s^2\), and every factor in the denominator
of \eqref{eq:global-difference} is positive. The right-hand side is therefore
nonnegative and vanishes only when \(r=s\). Consequently,
\begin{equation}
\label{eq:tau-y-star}
\tau^*=1,
\qquad
y^*=a.
\end{equation}

Substituting \eqref{eq:tau-y-star} into
\eqref{eq:u-minus-general} and \eqref{eq:u-plus-general} gives the exact
optimal control
\begin{equation}
\label{eq:exact-control}
u^*(t)
=
\begin{cases}
\dfrac{2ae^t}{e^2-1},
&0\leq t\leq1,\\[8pt]
\dfrac{2ae^{2-t}}{e^2-1},
&1<t\leq2.
\end{cases}
\end{equation}
Solving the corresponding linear state equations on the two sides of the
crossing and imposing continuity at \(t=1\) gives
\begin{equation}
\label{eq:exact-state}
x^*(t)
=
\frac{a\left(e^t-e^{2-t}\right)}{e^2-1},
\qquad
0\leq t\leq2.
\end{equation}
The exact trajectory satisfies
\begin{equation}
x^*(0)=-a,
\qquad
x^*(1)=0,
\qquad
x^*(2)=a.
\end{equation}

The optimal control attains its maximum at the crossing time:
\begin{equation}
u^*(1)
=
\frac{2ae}{e^2-1}.
\end{equation}
For \(a=0.2\),
\begin{equation}
u^*(1)
\approx0.1701836
<
u_{\max}=0.2.
\end{equation}
Thus, the control bound is inactive at the optimum, confirming that the
unconstrained minimum-effort construction above is admissible.

\section{Reinforcement-Learning Baseline Implementation Details}

The PPO and SAC comparisons were implemented with Gymnasium and
Stable-Baselines3. For each scalar benchmark, the reinforcement-learning
environment uses the same finite horizon, time step, RK4 transition map,
physical control bounds, and discrete objective used for the corresponding
HPDPO calculation. The observation contains the physical state together with
normalized time, and the trained policy is evaluated deterministically from
the prescribed initial condition. PPO and SAC were each trained for
\(250{,}000\) environment transitions for every reported random seed.

\subsection{Example 1: constructed nonsmooth benchmark}

For the constructed benchmark, the observation is
\((x_k,k/M)\), and the action space is the physical interval
\([0,0.2]\); no additional action rescaling is applied by the environment.
The training reward is the negative of the original one-step cost multiplied
by \(100\). This reward scaling affects optimization only; all reported
objectives and errors are evaluated in the original physical units.

PPO uses learning rate \(3\times10^{-4}\), rollout length \(2000\),
batch size \(100\), \(10\) optimization epochs per rollout,
discount factor \(\gamma=1\), GAE parameter \(0.95\), clipping range \(0.2\),
entropy coefficient \(0\), value-function coefficient \(0.5\), and maximum
gradient norm \(0.5\). The policy and value networks each contain two hidden
layers of \(64\) units.

SAC uses learning rate \(3\times10^{-4}\), replay-buffer capacity
\(250{,}000\), \(2000\) initial transitions before learning, batch size \(256\),
\(\gamma=1\), soft target-update coefficient \(\tau=0.005\), one gradient
step per environment step, and automatic entropy tuning. The policy network
contains two hidden layers of \(64\) units. All PPO and SAC runs are performed
on CPU with one PyTorch thread.

Five independent seeds, \(0,1,2,3,4\), are used for each algorithm.
Since one episode contains \(M=100\) control decisions, the
\(250{,}000\)-transition budget corresponds to \(2500\)
full-horizon-equivalent episodes per run. Table~\ref{tab:si-example1-rl}
summarizes the mean and sample standard deviation across the five seeds.

\begin{table}[ht]
\centering
\small
\caption{Five-seed PPO and SAC summary for the constructed nonsmooth
benchmark. HPDPO values are included for reference.}
\label{tab:si-example1-rl}
\begin{tabular}{lccc}
\hline
Quantity & HPDPO & PPO & SAC\\
\hline
Discrete objective \(J_h\)
& \(1.45\times10^{-2}\)
& \((1.4933\pm0.0147)\times10^{-2}\)
& \((1.5586\pm0.0428)\times10^{-2}\)\\
Control RMSE
& \(4.99\times10^{-4}\)
& \((1.732\pm0.201)\times10^{-2}\)
& \((3.409\pm0.723)\times10^{-2}\)\\
State RMSE
& \(2.12\times10^{-4}\)
& \((1.473\pm0.308)\times10^{-2}\)
& \((1.421\pm0.494)\times10^{-2}\)\\
Objective gap
& \(7.26\times10^{-8}\)
& \((4.52\pm1.47)\times10^{-4}\)
& \((1.10\pm0.43)\times10^{-3}\)\\
Crossing-time error
& \(1.83\times10^{-3}\)
& \((7.69\pm1.71)\times10^{-2}\)
& \((8.95\pm4.58)\times10^{-2}\)\\
Terminal-state error
& \(2.36\times10^{-4}\)
& \((2.34\pm0.61)\times10^{-2}\)
& \((1.27\pm0.95)\times10^{-2}\)\\
\hline
\end{tabular}
\end{table}

\subsection{Example 2: published nonsmooth benchmark}

For the published nonsmooth benchmark, the observation is again the physical
state together with normalized time, and the physical action lies directly in
\([-1,1]\). The reinforcement-learning training objective includes a soft
squared penalty on violations of the state interval \([0,1]\), with penalty
coefficient \(10\). This penalty is used only during training. The reported
objective and accuracy metrics are always recomputed from the original
unpenalized optimal-control problem, and state feasibility is checked
separately. The training reward is multiplied by a fixed factor of \(10\).

PPO uses learning rate \(3\times10^{-4}\), rollout length \(4000\), batch size
\(200\), \(10\) optimization epochs, \(\gamma=1\), GAE parameter \(0.95\),
clipping range \(0.2\), entropy coefficient \(0\), value-function coefficient
\(0.5\), and maximum gradient norm \(0.5\). The policy and value networks each
contain two hidden layers of \(64\) units.

SAC uses learning rate \(3\times10^{-4}\), replay-buffer capacity
\(250{,}000\), \(2000\) initial transitions before learning, batch size \(256\),
\(\gamma=1\), \(\tau=0.005\), one gradient step per environment step,
automatic entropy tuning, and two hidden layers of \(64\) units. Five
independent seeds, \(0,1,2,3,4\), are reported for each method. Since one
episode contains \(M=200\) control decisions, the training budget corresponds
to \(1250\) full-horizon-equivalent episodes per run.

Table~\ref{tab:si-example2-rl} reports the mean and sample standard deviation
over all five seeds. PPO produced feasible trajectories for two of the five
runs, whereas all five SAC runs satisfied the state constraint. Feasibility is
therefore reported separately rather than filtering the stochastic results
before averaging.

\begin{table}[ht]
\centering
\small
\caption{Five-seed PPO and SAC summary for the published nonsmooth benchmark.
HPDPO values are included for reference.}
\label{tab:si-example2-rl}
\begin{tabular}{lccc}
\hline
Quantity & HPDPO & PPO & SAC\\
\hline
Discrete objective \(J_h\)
& \(-0.661675\)
& \(-0.2434\pm0.1775\)
& \(-0.4629\pm0.1342\)\\
Control RMSE
& \(7.30\times10^{-17}\)
& \(0.4480\pm0.1027\)
& \(0.2945\pm0.1256\)\\
State RMSE
& \(4.56\times10^{-3}\)
& \(0.4343\pm0.0978\)
& \(0.2463\pm0.0622\)\\
Objective gap
& \(<10^{-14}\)
& \(0.4183\pm0.1775\)
& \(0.1987\pm0.1342\)\\
Feasible runs
& \(1/1\)
& \(2/5\)
& \(5/5\)\\
\hline
\end{tabular}
\end{table}

\section{Example 3: Learned ReLU Approximation of the Four-State Cascade}

The network used to approximate the untreated four-state vector field is
trained directly on random states from
$[0,1]\times[0,1.5]\times[0,1.5]\times[0,1.5]$. With the default settings,
$100{,}000$ state samples are generated and split after shuffling into
$80\%$ training, $10\%$ validation, and $10\%$ test subsets. Inputs and
right-hand-side targets are standardized using statistics computed from the
training subset only.

The network architecture is
$4\rightarrow256\rightarrow256\rightarrow4$ with ReLU activation functions.
The untreated vector field $f_{\mathrm{true}}(x,0)$ is the supervised target;
the treatment term $(-uA,0,0,0)^T$ is not learned and is added explicitly in
the controlled model. Training uses AdamW, learning rate $10^{-3}$, weight
decay $10^{-6}$, batch size $2048$, a maximum of $4000$ epochs, gradient-norm
clipping at $1$, and a cosine-annealing learning-rate schedule. The loss is a
component-weighted mean squared error with default weights $(1,1,1,1)$.
The checkpoint with the smallest validation loss is retained. The default
random seed is $123$, and computations use double precision.

The default training set contains only random feasible-box samples. Optional
untreated or treated trajectory states can be added by command-line options,
but their targets remain the untreated vector field. The trajectory plots in
the training script are post-training validation calculations and are not
included in the default fitting set.

\section{Example 4: ADNI Amyloid Neural ODE Training Details}

The revised amyloid training implementation reads normalized amyloid values
from the data processing worksheet. Subjects with at
least two usable amyloid visits are retained. The split is performed at the
subject level with fractions $0.70/0.15/0.15$ for training, validation, and
test data using random seed $42$. For the 315 eligible subjects reported in
the manuscript, this rule gives 220 training subjects, 47 validation subjects,
and 48 test subjects. The test subset is not evaluated during fitting; it is
evaluated only after the validation-selected checkpoint has been restored.

The learned right-hand side is
\[
 g_{\phi}(A)=\left[\operatorname{Softplus}_{\beta=2}
 (h_{\phi}(A))\right]A(1-A),
\]
The network $h_{\phi}$ has two hidden ReLU layers with
$64$ units per layer. The population anchor $A(0)$ is optimized jointly with
the network. Its initialization is the median of training observations closest
to $s=0$, using at least 10 observations and otherwise approximately the
closest $5\%$ of training observations.

Training uses Adam with learning rate $10^{-3}$ and weight decay $10^{-6}$.
The maximum number of epochs is $1200$, gradient norms are clipped at $5$, and
early stagnant validation performance triggers early stopping after $150$
epochs without an improvement larger than $10^{-8}$. A
ReduceLROnPlateau scheduler halves the learning rate after sustained validation
stagnation and does not reduce it below $10^{-6}$. Neural ODE trajectories are
integrated by a differentiable RK4 scheme with maximum step size $0.02$. The computation uses double precision.

The fitted objective in the revised implementation is the subject-balanced
MSE used in the main paper plus a weak anchor penalty and a weak squared-rate
penalty. The anchor penalty weight is $10^{-4}$. The rate penalty weight is
$10^{-6}$ and is evaluated on 128 uniformly spaced amyloid states over
$[0,1]$. Model selection is based only on validation subject-balanced MSE.
The script records split assignments, fit metrics, training history,
predictions, and the selected checkpoint to support reproducibility.

\section{Example 4: Data-Driven Amyloid Control and Smooth Reference Settings}

The data-driven control comparison uses the learned checkpoint directly over
the DPS interval $[0,10]$ with step size $0.1$. The controlled model is
$dA/ds=g_{\phi}(A)-\kappa uA$ with $\kappa=1$, treatment penalty $\rho=0.1$,
and $0\le u\le1$. The initial control is $0.1$. The HPDPO implementation uses
12 outer iterations, one control sweep and one multiplier sweep per outer
iteration, 101 candidate actions, $\beta_u=10^{-4}$,
$\beta_{\lambda,s}=10$, $\beta_{\lambda,2}=0.05$, multiplier bound $50$, and
unit relaxation factors. The stopping tolerances for control and multiplier
changes are both $10^{-8}$.

The smooth reference is obtained from a degree-9 Legendre approximation of the
learned right-hand side using 2001 fitting points, ridge parameter $10^{-12}$,
and endpoint weight $100$. The FBS reference uses at most 500 iterations,
relaxation factor $0.5$, stopping tolerance $10^{-8}$, and initial control
$0.1$. The random seed for this comparison is $42$, and the calculations use
double precision.

\section{Multidimensional Example 1: Two-State, Two-Control Benchmark}

The compartmental benchmark uses \(T=1\), \(\Delta t=0.01\), and
\(M=100\) policy intervals. The state trajectory is propagated by classical
RK4 with the same time step. The two control components are optimized jointly
at every policy coordinate. A \(21\times21\) Cartesian-product grid over the
full admissible rectangle is evaluated first, followed by two
\(11\times11\) local joint refinements around the current best action.

The reported HPDPO calculation uses \(12\) outer iterations,
\(\beta_u=0\), \(\beta_{\lambda,s}=1\),
\(\beta_{\lambda,2}=0.05\), multiplier bound \(25\), and a backward
multiplier sweep per outer iteration. A relaxation factor \(0.20\) is used in
the numerical multiplier update. The initial control is the interior constant
pair \((\alpha,\beta)=(2.5,2.5)\), and the multiplier sequence is initialized
to zero. The known switching time is not supplied to the optimizer.

The global and local candidate evaluations always vary both control
components simultaneously. Thus, the multidimensional calculation is a joint
two-dimensional action search at each time point rather than two independent
component-wise updates.

\section{Multidimensional Example 2: Five-State CMO Alzheimer's Disease Model}

For the controlled monomer--oligomer model, the treatment policy is
piecewise constant over intervals of \(\Delta t=5\,{\rm s}\) on
\(T=1000\,{\rm s}\), giving \(200\) policy intervals. Within every policy
interval, the five-state nonlinear dynamics and the accumulated running cost
are integrated together by RK4 using internal substeps no larger than
\(0.5\,{\rm s}\). This ensures that the running-cost quadrature used during
candidate selection is the same as that used for the final reported objective.

Both treatment components are varied jointly. The reported calculation uses a
\(21\times21\) global Cartesian-product grid followed by two
\(11\times11\) local joint refinements. HPDPO is run for \(12\) outer
iterations with no additional control-smoothing penalty,
\(\beta_{\lambda,s}=1\), \(\beta_{\lambda,2}=0.05\), multiplier bound
\(25\), three backward multiplier sweeps per outer iteration, and multiplier
relaxation factor \(0.20\). The treatment sequence is initialized at the
constant interior pair \((v,w)=(0.01,0.01)\), and the multipliers are
initialized to zero.

The no-treatment and maximum-treatment objectives are evaluated with the same
forward solver as implementation checks against the published CMO study. The
main comparison is made with the published forward--backward numerical
solution; it is not interpreted as comparison with a known analytical global
optimum.

\section{Multidimensional Example 3: Six-State, Five-Control Flood-Management Problem}

The flood-management benchmark uses \(T=12\) and \(\Delta t=0.1\), giving
\(120\) policy intervals. State propagation uses classical RK4 on this common
time grid. Because a tensor-product action grid is impractical in five control
dimensions, the candidate set is constructed by a dimension-dependent
Sobol/QMC search.

At each policy coordinate, the reported global search uses \(512\) scrambled
Sobol candidates over the full five-dimensional admissible control box,
augmented by all \(2^5=32\) box corners and the current action. Two local
five-dimensional refinements are then performed, each using \(256\) Sobol
candidates over a shrinking hyperrectangle around the current best action.
The initial local radius is \(0.20\) times the componentwise control range and
is multiplied by \(0.35\) after each refinement. The reproducible QMC base seed
is \(20260921\); the implementation deterministically changes the scrambled
design across outer iterations, policy coordinates, and local refinement
levels.

The HPDPO calculation uses \(12\) outer iterations,
\(\beta_v=0\), \(\beta_{\lambda,s}=1\),
\(\beta_{\lambda,2}=0.05\), multiplier bound \(25\), three backward
multiplier sweeps per outer iteration, and multiplier relaxation factor
\(0.20\). The control sequence is initialized at the midpoint of each
admissible control interval and the multiplier sequence is initialized to
zero. All five control components are varied jointly in every candidate
evaluation.

For the reference calculation, the HJB auxiliary system is solved backward
with an adaptive RK45 integrator using relative tolerance \(10^{-11}\),
absolute tolerance \(10^{-13}\), and maximum internal step \(0.005\). The
resulting HJB control is then sampled on the same \(\Delta t=0.1\) grid used by
HPDPO so that objective, state-RMSE, and control-RMSE comparisons use a common
temporal discretization.

\section{Machine-Learning and Numerical Reporting Summary}

\begin{table}[ht]
\centering
\small
\caption{Summary of the principal learning and numerical settings used in the
reported experiments.}
\begin{tabular}{p{0.28\textwidth}p{0.62\textwidth}}
\hline
Item & Reported implementation \\
\hline
Four-state ReLU data
& \(100{,}000\) feasible-box samples; \(80/10/10\) train/validation/test split; training-only normalization.\\
Four-state ReLU model
& \(4\)--\(256\)--\(256\)--\(4\) ReLU network; AdamW; batch \(2048\);
learning rate \(10^{-3}\); weight decay \(10^{-6}\); maximum \(4000\)
epochs; seed \(123\).\\
ADNI split
& Subject-level \(70/15/15\); seed \(42\); test set evaluated only after
validation-based checkpoint selection.\\
ADNI Neural ODE
& Two hidden ReLU layers with \(64\) units each; Softplus-constrained RHS;
Adam; learning rate \(10^{-3}\); weight decay \(10^{-6}\); maximum \(1200\)
epochs; differentiable RK4 maximum step \(0.02\).\\
Scalar PPO
& Stable-Baselines3 MLP actor--critic; two \(64\)-unit hidden layers;
learning rate \(3\times10^{-4}\); \(\gamma=1\); five seeds; \(250{,}000\)
transitions per seed.\\
Scalar SAC
& Stable-Baselines3 MLP; two \(64\)-unit hidden layers; learning rate
\(3\times10^{-4}\); batch \(256\); \(\gamma=1\); \(\tau=0.005\);
automatic entropy coefficient; five seeds; \(250{,}000\) transitions per seed.\\
Two-control HPDPO search
& \(21\times21\) global joint grid plus two \(11\times11\) local joint
refinements.\\
Five-control HPDPO search
& \(512\) global Sobol candidates plus \(32\) corners; two local refinements
with \(256\) Sobol candidates each; all five control components varied jointly.\\
\hline
\end{tabular}
\end{table}

\clearpage
\renewcommand{\bibsection}{\section*{\refname}}
\bibliographystyle{unsrtnat}
\bibliography{references_rl_additions_final}
\end{document}